\documentclass{article}
\usepackage{a4wide,amsmath,amssymb,amsthm}
\usepackage{bbm}
\usepackage[numbers]{natbib}
\usepackage{graphicx}
\usepackage{url}
\usepackage{subcaption}

\newcommand{\mR}{\mathbb{R}}

\renewcommand{\phi}{\varphi}

\newcommand{\one}{\mathbbm{1}}
\newcommand{\ki}{^{(k)}}
\newcommand{\kn}{^{(k+1)}}
\newcommand{\pOT}{\phi}
\newcommand{\preg}{\phi_\eta}

\newcommand{\ed}[1]{{\rm #1}}

\newtheorem{theorem}{Theorem}[section]
\newtheorem{lemma}[theorem]{Lemma}

\newtheorem{corollary}[theorem]{Corollary}
\newtheorem{proposition}[theorem]{Proposition}
\theoremstyle{definition}
\newtheorem{definition}[theorem]{Definition}

\newtheorem{remark}[theorem]{Remark}
\newtheorem{example}[theorem]{Example}

\newtheorem*{example*}{Example}

\renewcommand{\baselinestretch}{1.5}
\let\origfootnote\footnote
\renewcommand{\footnote}[1]{%
   \begingroup
   \renewcommand{\baselinestretch}{1}%
   \origfootnote{#1}%
   \endgroup}

\begin{document}

\title{A Multi-Objective Interpretation of Optimal Transport}

\author{%
Johannes M. Schumacher%
\thanks{\baselineskip12pt University of Amsterdam, Faculty of Economics and Business, Section Quantitative Economics;
{\tt j.m.schumacher@uva.nl}.}}

\maketitle
\begin{abstract}
\noindent
This paper connects discrete optimal transport to a certain class of multi-objective optimization problems. In both settings, the decision variables can be organized into a matrix. In the multi-objective problem, the notion of Pareto efficiency is defined in terms of the objectives together with non-negativity constraints and with equality constraints that are specified in terms of column sums. A second set of equality constraints, defined in terms of row sums, is used to single out particular points in the Pareto efficient set which are referred to as ``balanced solutions''\!. Examples from several fields are shown in which this solution concept appears naturally. Balanced solutions are shown to be in one-to-one correspondence with solutions of optimal transport problems. As an example of the use of alternative interpretations, the computation of solutions via regularization is discussed.
\end{abstract}

\section{Introduction} \label{introduction}

The theory of optimal transport is an extremely rich research area, with applications ranging from probability theory to mathematical physics, from image processing to partial differential equations, and from city planning to differential geometry. The two volumes by Rachev and R\"{u}schendorf \cite{Rachev1,Rachev2} and the celebrated books by Fields medalist Villani \cite{Villani03,Villani} provide ample evidence of the spectacular growth of the field since the original work of Monge \cite{Monge}, Kantorovich \cite{Kantorovich}, and Fr\'echet \cite{Frechet}. Specialized treatments for applied mathematicians and for economists, respectively, are provided in the
recent books by Santambrogio \cite{Santambrogio} and Galichon \cite{Galichon}.

The purpose of the present paper is to highlight a connection between discrete optimal transport and a certain class of multi-objective problems. This connection appears to have gone unnoticed so far, even though the proof is not difficult, and the corresponding result for strictly concave problems has been known for several decades \cite{Gale77,Gale79,Buehlmann79}. By means of this connection, the techniques of optimal transport can be applied to certain multi-objective problems, while at the same time an alternative perspective on optimal transport is obtained. As an illustration of the latter point, it will be shown below how the popular entropic regularization of optimal transport problems \cite{Cuturi,Benamou} has a natural interpretation as an isoelastic regularization within the multi-objective framework. The transfer from a single objective to multiple objectives also entails a change from an additive to a multiplicative environment, and in this way leads naturally to applications of nonlinear Perron-Frobenius theory. This theory has been used in \cite{Brualdi} (cf.\ also \cite{Lemmens}) to analyze a matrix scaling problem that can be related to single-objective concave optimization, and it comes up perhaps even more naturally in the context of multi-objective problems \cite{PSW1}. The application of nonlinear Perron-Frobenius theory will be illustrated below.

In this paper, the term ``multi-objective problem'' is used to cover multi-agent problems (several agents, each with their own objective), as well as multicriteria or vector optimization problems (one agent with several objectives). The solution concept that will be employed can be used in both contexts, as illustrated in Section \ref{examples}. The examples in this section also show the relevance of the concept in different fields: operations research, statistics, and actuarial science.

From the multi-agent point of view, the multi-objective problems considered in this paper may be placed within the broad class of fair division problems \cite{Brams,Moulin}. Typical features of problems considered in this active area of research include multiple agents, each equipped with their own utility functionals, and resources in limited supply. Solution concepts are based on notions of \emph{efficiency} and \emph{fairness}. While many different notions of fairness have been proposed, the one that appears to be used most often is envy-freeness: no agent would prefer anybody else's share to her own \cite{Foley}. In the present paper, a different concept of fairness is used, which is less standard in the literature on fair division, and which is based on bringing in a notion of \emph{size} that is independent of any particular agent. This fairness concept can be viewed as a version of ``exogenous rights''\!, mentioned as one of four principles of distributive justice by Moulin \cite[Ch.\,2]{Moulin}, or of ``entitlements'' as discussed by Brams \cite{Brams}. In contrast to the use of entitlements in \cite{Brams}, the size constraint is taken in this paper as an equality constraint, rather than as an inequality constraint. It will be assumed below that allotments to agents can be represented as vectors, and that size is expressed as a single number determined by a linear functional.

Size constraints can be used as a way to single out unique solutions among the set of all efficient solutions. In that sense, their role is comparable to, for instance, the Nash collective utility function \cite{Nash}. The use of size constraints presupposes the availability of a measurement functional by which the magnitudes of allotments can be determined. In this sense, size constraints call for a richer problem environment than the Nash bargaining solution does. In addition to the availability of a measurement functional, the problem data should provide, for each of the agents/objectives, the value of the right-hand side of the equation that expresses the corresponding size constraint. Below, these numbers appear as row sums that are prefixed.

In the literature on multicriteria optimization (see for instance \cite{Ehrgott}), several methods have been proposed to select particular solutions among the set of all Pareto efficient solution, such as scalarization and lexicographic ordering. Scalarization is a way to relate a multi-objective problem to a single-objective problem. While this paper also establishes a connection between single-objective and multi-objective problems, and weighted-sum scalarization will be used as part of the proof technique, it should be emphasized that the connection established below cannot be viewed as a scalarization. The objective function in the associated single-objective problem cannot be obtained through aggregation of the objective functions in the multicriteria problem by means of an achievement function, as in
\cite[Section 4.6]{Ehrgott}. Multi-objective optimal transport has been studied for instance by Isermann \cite{Isermann79}, but again this is different from the problem considered in this paper, which is to connect optimal solutions of single-objective optimal transport problems to balanced solutions (as introduced below) of a class of multi-objective problems.

In the context of equilibrium economics, efficient solutions that satisfy budgetary size constraints have been termed \emph{BCPE solutions} (budget-constrained Pareto efficient solutions) by Ba\-lasko \cite{Balasko}. Gale and Sobel \cite{Gale79} do not introduce a separate term for the combination of Pareto efficiency and satisfaction of budget constraints; they speak of \emph{optimal and fair} solutions (\emph{optimal and proportional} in \cite{Gale82}). B\"{u}hlmann and Jewell \cite{Buehlmann79} use the term \emph{fair Pareto optimal risk exchanges} (FAIRPOREX). Neither of these terminologies appears to be entirely appropriate to cover the range of different situations exemplified below. For generality and simplicity, the term \emph{balanced solutions} will be used in this paper.\footnote{It is difficult to find expressions that have not already been used elsewhere. The term ``balanced outcome'' has been used by \cite{Kleinberg} to refer to an extension of the Nash bargaining solution to networks, which is different from but somewhat related to the notion used in this paper. In the Shapley-Bondereva theorem (see for instance \cite{Myerson}), balancedness is used in a quite different sense. In connection with the transportation problem, the term ``balanced'' is sometimes used to refer to a property that is called the global feasibility condition in this paper; see (\ref{gfc}).}

Individual rationality and behavioral effects do not play a role in the definition of balanced solutions. In some (but not all) of the examples shown below, it would in fact be quite unnatural to introduce these notions. All in all, it may perhaps be said that the concept of a balanced solution is closer in spirit to the multicriteria setting than to the multi-agent setting. In the case where multiple agents are involved, the concept represents a social planner's  point of view.

The organization of the paper is as follows. Following this introduction, a formal definition of the notion of balanced solutions is given in Section \ref{sec_balanced}, along with three examples of situations in which this notion may be used. The same section also provides a characterization of balanced solutions. This characterization is used in Section \ref{connection} to establish the connection to optimal transport. Regularization and the connection to Perron-Frobenius theory are discussed in Section \ref{iterative}. A discussion
of numerical issues follows in Section \ref{numerics}, and Section \ref{conclusions} concludes.

\section{Balanced solutions of linear multi-objective problems} \label{sec_balanced}

\subsection{Definitions}

In the optimal transport problem, the decision variables are organized into a matrix, and there are constraints related to row sums, as well as constraints related to column sums. Rows and columns typically correspond to objects that are different in nature, for instance producers and consumers. The problem is formulated with a single objective. As will be illustrated below, situations arise in practice which similarly call for a table to be filled, taking into account row sum constraints, as well as column sum constraints, but which are different in the sense that there are multiple objectives, which relate (by convention) to the rows of the matrix that is to be formed. Additionally, as shown in the examples below, there is often a certain status difference between the row constraints and the column constraints. The column constraints express ``feasibility'', whereas the row constraints express a notion of ``fairness''\!. While the column constraints can be thought of being imposed by Nature, the row constraints are man-made, and the right-hand sides used in these constraints are to some extent a matter of choice. Such a hierarchy leads to the notion of what we will call ``balanced'' solutions of the multi-objective problem. The idea is that the objective functions together with the feasibility constraints are used to define a Pareto efficient surface, and that the row constraints are then used to select specific solutions on the efficient surface. Under suitable circumstances, it may happen that there is a unique balanced solution. A precise definition is given below. In the interest of generality, the terms ``primary'' and ``secondary'' are used to refer to the two types of constraints, rather than ``feasibility'' and ``fairness''\!.
\vskip2mm
\begin{definition} \label{moma}
A \emph{multi-objective matrix allocation} (MOMA) problem is specified by
\begin{itemize}
\item[(i)] the decision variables $x_{ij}$, which together form a matrix $X$ of size $n \times m$;
\item[(ii)] objective functions given by
$$
\Psi_i(X_{i\cdot}) = \sum_{j=1}^m g_{ij}(x_{ij}) \qquad (i=1,\dots,n)
$$
\item[(iii)] primary constraints given by
$$
\sum_{i=1}^n x_{ij} = c_j \quad (j=1,\dots,m), \qquad x_{ij} \geq 0 \quad (i=1,\dots,n,\; j = 1,\dots,m)
$$
where $c_1,\dots,c_m$ are given positive constants;
\item[(iv)] secondary constraints given by
$$
\sum_{j=1}^m x_{ij} = r_i \qquad (i=1,\dots,n)
$$
where $r_1,\dots,r_n$ are given positive constants.
\end{itemize}
\end{definition}
\vskip2mm
A specification of the problem above can be given in terms of the triple $(g,c,r)$, where $g = (g_{ij}(\cdot))_{i=1,\dots,n;\,j=1,\dots,m}$ is the matrix of reward functions, $c \in \mR^m_{++}$ is the vector of given column sums, and $r \in \mR^n_{++}$ is the vector of given row sums.\footnote{The notation $\mR_{++}$ is used in this paper to indicate the set of strictly positive real numbers. The set of non-negative real numbers is denoted by $\mR_+$.}
The functions $g_{ij}: \mR_+ \rightarrow \mR \cup \{-\infty\}$ will be assumed to be twice continuously differentiable (for convenience), strictly increasing, and concave. An obvious necessary condition for a MOMA problem to be feasible is that the sum of the row sums equals the sum of the column sums:
\begin{equation}\label{gfc}
\sum_{i=1}^n r_i = \sum_{j=1}^m c_j.
\end{equation}
This condition will be referred to as the \emph{global feasibility condition}. We take it as a standing assumption in all MOMA problems to be discussed below.

\begin{definition} \label{balanced}
An $n \times m$ matrix $X$ is said to provide a \emph{balanced solution} of a given MOMA problem if it satisfies the following two conditions.
\begin{itemize}
\item[(i)] The decision matrix $X$ is Pareto efficient with respect to the objective functions $\Psi_i$ and the primary constraints of the MOMA problem. In other words, there does not exist a matrix $\tilde{X}$ such that $\tilde{X}$ satisfies the primary constraints, $\Psi_i(\tilde{X}_{i\cdot}) \geq \Psi_i(X_{i\cdot})$ for all $i=1,\dots,n$, and $\Psi_i(\tilde{X}_{i\cdot}) > \Psi_i(X_{i\cdot})$ for at least one $i \in \{1,\dots,n\}$.
\item[(ii)] The matrix $X$ satisfies the secondary constraints of the MOMA problem.
\end{itemize}
\end{definition}

\begin{remark}
The formulation as given above refers to the point of view of maximization. An analogous formulation can be given for minimization problems. In the case of minimization, the decision variables represent quantities that are preferred to be small, at least from the perspective of the $i$-th objective, such as ``payment to be made by agent $i$ if outcome $j$ occurs'' or ``part of chore $j$ to be carried out by agent $i$''\!. For such problems, natural assumptions on the functions $g_{ij}$ (cost functions in this case, rather than reward functions) are that they are increasing and convex.
\end{remark}

\begin{remark}
In the above, we work with \emph{unweighted} sums. A similar problem with weighted row and column sums
$$
\sum_{i=1}^n p_i x_{ij} = c_j , \qquad \sum_{j=1}^m q_j x_{ij} = r_i
$$
($p_i > 0$, $q_j > 0$ for all $i$ and $j$) can be transformed to a problem with unweighted sums by the substitutions
$$
\tilde{x}_{ij} = p_i q_j x_{ij}, \quad \tilde{c}_j = q_j c_j, \quad \tilde{r}_i = p_i r_i.
$$
In this sense, the use of unweighted sums entails no loss of generality. On the other hand, it might be argued that such usage may not offer the best preparation for the continuous case. For simplicity of notation, however, in this paper unweighted sums will be used.
\end{remark}

\begin{remark} \label{inv}
If the reward functions $g_{ij}(x_{ij})$ are modified to
$$
\tilde{g}_{ij}(x_{ij}) = s_i g_{ij}(x_{ij}) + t_{ij}
$$
where $s_i > 0$ and $t_{ij}$ are constants ($i=1,\dots,n$; $j=1,\dots,m$), then any balanced solution of the original problem is also a balanced solution of the modified problem, and vice versa. This is immediate from the definition of balanced solutions, and from the additive separability of the objective functions.
\end{remark}

\begin{remark}
If $X$ is a balanced solution of the MOMA problem specified by $(g,c,r)$, and $s$ is any positive number, then $\tilde{X} := X/s$ is a balanced solution of the problem specified by $(\tilde{g},r/s,c/s)$ with $\tilde{g}_{ij}(x_{ij}):=g_{ij}(sx_{ij})$, and vice versa.
Under the global feasibility constraint, it is therefore no restriction of generality to assume that $\sum_i r_i = \sum_j c_j = 1$.
\end{remark}

A \emph{linear} MOMA problem is one in which the reward functions are linear, i.e.\ $g_{ij}(x_{ij}) = b_{ij}x_{ij}$ with constant coefficients $b_{ij}>0$.

\subsection{Examples} \label{examples}

Three examples are given of situations in which MOMA problems, and in particular linear MOMA problems, arise naturally. These examples refer to task allocation, statistical classification, and risk sharing respectively.

\begin{example} \label{task}
Mathematical formalization of the problem of chore division appears to have been first undertaken by Martin Gardner \cite{Gardner}. The problem is naturally formulated as a minimization problem, and the exogenous rights that were mentioned in Section \ref{introduction} are in this case exogenous obligations. Consider a situation in which $n$ agents must carry out $m$ different tasks. For an example that may be familiar to many readers, one can think of $m$ courses, which need to be taught by $n$ available lecturers. The size of the contribution to be made by each of the agents is supposed to be prefixed; in the example of course allocation, this corresponds to a teaching load which is given in advance for each lecturer. The contribution size is computed by means of a linear formula, in which each of the tasks has its own nominal weight. The decision variable $x_{ij}$ represents the fraction of task $j$ that is assigned to agent $i$. The objective functions (cost in this case) are the disutilities of the agents that result from their package of assigned tasks. The primary constraints express that the task fractions must be non-negative, and that each task must be fulfilled completely. The secondary constraints correspond to the predetermined magnitudes of the contributions of the agents. The column sums are all 1 (i.e.\ 100\%), and the row sums are given by the teaching loads. In this way, a MOMA problem is defined. If the agents' disutilities depend linearly on the task fractions, then we have a linear MOMA problem.
\end{example}

\begin{example} \label{classification}
Statistical classification is often considered as part of the field of machine learning \cite{Michie}. The central problem in this area
is to place objects into mutually exclusive categories, on the basis of observed characteristics. Examples of such situations are numerous and include, for instance, various types of diagnosis. To stay within the discrete context, it will be assumed here that only a finite number $m$ of different characteristics can be observed (in other words, the ``feature space'' is finite), and that there is a finite number $n$ of categories. The probability that a given object has characteristic $j$ and belongs to category $i$ will be denoted by $p_{ij}$. It is assumed that these probabilities are known. The situation in which $n=2$ corresponds to the problem of testing a simple null hypothesis against a simple alternative. Already in this case, the multiple-objective nature of the problem is apparent; the probability of a type-I error must be weighed against the probability of a type-II error. For a general number of categories, we can define with each category an objective function which is given by the probability of correct classification for an object in that category. Decisions may be randomized; we denote by $x_{ij}$ the probability by which the classifier will place an object with characteristic $j$ into category $i$. A problem of MOMA type arises when the percentage of objects to be placed in each of the categories is determined in advance. Such a prescription may be reasonable in circumstances in which classification is carried out for the purpose of treatment (for instance, medical treatment, or participation in an educational program), and there are capacity constraints for each form of treatment. The decision variables are given by the numbers $x_{ij}$. The probability that an object will be classified into category $i$, given that it indeed belongs to that category, is given by $\sum_{j=1}^m p_{ij}x_{ij}/\sum_{j=1}^m p_{ij}$ ($i=1,\dots,n$). These probabilities may be taken as objective functions; in the case of two categories, they are known as \emph{specificity} and \emph{sensitivity}, and are in a one-to-one relationship with the probabilities of type-I errors and type-II errors. The primary constraints are given by the non-negativity constraints on the decision variables $x_{ij}$, and by the condition that all objects with a given feature $j$ should be classified into one of the categories ($\sum_{i=1}^n x_{ij} = 1$ for all $j$). The secondary constraints correspond to the prescribed percentages of objects to be placed in each category. Since the objective functions are linear in this case, we have a linear MOMA problem.
\end{example}

\begin{example} \label{revenue}
Risk sharing is a classical topic in actuarial science (see for instance \cite{Rotar}) and remains a very active field of research. While part of the literature on risk sharing is focused on information asymmetries and behavioral effects which we do not consider here, the problem has been formulated also in the MOMA framework \cite{Buehlmann79}. For a simple illustration,
consider $n$ individuals who form a group that will carry out a project with an uncertain payoff (say, a fishing expedition). For the purpose of discreteness, assume that there are $m$ possible outcomes. The variable $x_{ij}$ denotes the share of revenue that agent $i$ will receive in case outcome $j$ materializes. The right-hand side of the column constraint $\sum_{i=1}^n x_{ij} = c_j$ refers to the total amount that is available in outcome $j$. It would also be reasonable to impose the non-negativity constraint $x_{ij} \geq 0$. The decision variables $x_{ij}$ are to be determined at the beginning of the project, before the ``veil of uncertainty'' has been lifted. To express a notion of fairness, the group might agree on state prices $q_j$ for each of the possible outcomes, and fix a value for the allotment to each agent, i.e.\ $\sum_{j=1}^m q_j x_{ij} = v_i$, where the numbers $v_i$ are chosen in relation to the contributions (in terms of effort and money) made by participants. If the participants are risk neutral, i.e.\ their objectives are given by their subjective expected payoffs (in other words, $\Psi_i(X_{i\cdot}) = \sum_{j=1}^m p_{ij}x_{ij}$ where the coefficients $p_{ij}$ refer to subjective probabilities), then we obtain a linear MOMA problem. In a more actuarial context, one might consider a group of companies which agree to a mutual reinsurance agreement with respect to claims that will come in during a given period. Fairness may then for instance be expressed on the basis of the expected claim sizes of the portfolios that the companies bring into the pool. In this case, the problem is more naturally formulated as a minimization problem, since the agents prefer smaller realized claims to larger ones.
\end{example}

\subsection{Characterization of balanced solutions} \label{character}

Consider a linear MOMA problem given by
\begin{subequations} \label{linmoma}
\begin{align}
\text{objectives} \quad & \sum_{j=1}^m b_{ij}x_{ij} & (i=1,\dots,n) \hspace{21.4mm}& \hspace{2cm} \\[2mm]
\text{primary constraints} \quad & \sum_{i=1}^n x_{ij} = c_j & (j=1,\dots,m) \hspace{20mm} \\[2mm]
& \hspace{6.2mm} x_{ij} \geq 0 & (i=1,\dots,n;\; j = 1,\dots,m) \\[2mm]
\text{secondary constraints} \quad & \sum_{j=1}^m x_{ij} = r_i & (i=1,\dots,n) \hspace{21.6mm}
\end{align}
\end{subequations}
It will be assumed throughout that the coefficients $b_{ij}$ and the prescribed row and column sums $r_i$ and $c_j$ are positive, and that the global feasibility condition (\ref{gfc}) is satisfied. Under these conditions, the feasible set in decision space corresponding to the primary constraints is nonempty. The expressions above can refer to maximization as well as to minimization. In both cases, the coefficients $b_{ij}$ are supposed to be positive. Of course, the two cases differ by a reversal of the direction of the inequality used in the definition of Pareto efficiency.

It follows from a theorem of Isermann \cite{Isermann74} (see also \cite[Thm.\,6.6]{Ehrgott}) that a matrix $X$ is a Pareto efficient solution of the problem consisting of the objectives and the primary constraints, taken as a maximization problem, if and only if it is a solution of a weighted-sum optimization problem of the form
\begin{subequations} \label{ws}
\begin{align}
\hspace{2cm} \text{maximize} \quad & \sum_{i=1}^n \alpha_i \sum_{j=1}^m b_{ij}x_{ij}\\[2mm]
\text{subject to} \quad & \sum_{i=1}^n x_{ij} = c_j & (j=1,\dots,m) \hspace{20mm} \\[2mm]
& x_{ij} \geq 0 & (i=1,\dots,n;\; j = 1,\dots,m) & \hspace{2cm}
\end{align}
\end{subequations}
where $\alpha_1,\dots,\alpha_n$ are \emph{positive} parameters. As a result, we obtain the following characterization of balanced solutions of linear MOMA problems.

\begin{lemma} \label{charlemma}
A matrix $X \in \mR_+^{n \times m}$ is a balanced solution of the linear \ed{MOMA} problem specified by \ed{(\ref{linmoma})}, taken as a maximization problem, if and only if there exist positive numbers $\alpha_1,\dots,\alpha_n$ (weights) and positive numbers $\beta_1,\dots,\beta_m$ (dual variables relating to equality constraints), such that the following conditions are satisfied:
\begin{subequations} \label{char}
\begin{align}
\big(x_{ij}=0 \text{ and } \alpha_i b_{ij} \leq \beta_j\big) \text{ or }
\big(x_{ij} \geq 0 \text{ and } \alpha_i b_{ij} = \beta_j\big)  \quad  (i=1,\dots,n;\; j = 1,\dots,m) \label{momacomp} \\[2mm]
\sum_{i=1}^n x_{ij} = c_j \quad (j=1,\dots,m) \hspace{4cm} \label{cc} \\
\sum_{j=1}^m x_{ij} = r_i \quad (i=1,\dots,n). \hspace{4.05cm}
\end{align}
\end{subequations}
\end{lemma}

\begin{proof}
The weighted-sum optimization problem (\ref{ws}) is a linear program with equality and inequality constraints. By linear programming duality (see for instance \cite{Schrijver98}), a matrix $X$ is optimal for (\ref{ws}) if and only if there exist $\beta_j \in \mR$ such that the complementary slackness conditions (\ref{momacomp}) and the column constraints (\ref{cc}) are satisfied. The set of all Pareto efficient solutions of the multi-objective linear program specified by the objective functions and the primary constraints of (\ref{linmoma}) is therefore given by matrices $X \in \mR_+^{n \times m}$ for which there exist $\alpha_i > 0$ and $\beta_j \in \mR$ such that (\ref{momacomp}) and (\ref{cc}) hold. By the standing assumption that $b_{ij}>0$, this can only happen for dual variables which are positive.
\qed
\end{proof}

\begin{remark}
An analogous statement to Lemma \ref{charlemma} can be proven in the case of minimization; the only difference lies in the direction of the inequality in the first term of the disjunction in (\ref{momacomp}). Because of the constraints $x_{ij} \geq 0$ and $\sum_{i=1}^n x_{ij} = c_j > 0$, for every index $j$ there must be a least one index $i$ such that $x_{ij} > 0$; for this index $i$, the equality $\alpha_i b_{ij} = \beta_j$ holds. It follows that, also in the MOMA minimization case, the dual variables $\beta_j$ are positive.
\end{remark}

\section{Connection to optimal transport} \label{connection}

MOMA maximization problems with increasing and strictly concave utility functions were considered by \cite{Gale79}. It is shown in this paper that there is a connection between balanced solutions of such problems and optimal solutions of related single-objective matrix allocation (SOMA) problems of the form
\begin{subequations} \label{soma}
\begin{align}
\sum_{i=1}^n \sum_{j=1}^m f_{ij}& (x_{ij}) \rightarrow \text{ max} \\
\text{subject to} \quad \sum_{i=1}^n x_{ij} & = c_j \quad (j = 1,\dots,m) \\
\sum_{j=1}^m x_{ij} & = r_i \quad (i = 1,\dots,n) \\
x_{ij} & \geq 0.
\end{align}
\end{subequations}
Specifically, under regularity conditions, it was shown by \cite{Gale79} that a matrix $X$ is a balanced solution of the MOMA problem of Def.\,\ref{moma} if and only if $X$ is an optimal solution of the SOMA problem (\ref{soma}), where the reward functions $f_{ij}$ of the SOMA problem relate to the utility functions $g_{ij}$ of the MOMA problem by the prescription
\begin{equation} \label{momasoma}
f_{ij}'(x) = \log g_{ij}'(x).
\end{equation}
In other words, if one starts with the utility functions of the SOMA problem, then the utility functions of the corresponding MOMA problem are found from these by executing the following operations in succession: differentiation, exponentiation, integration. If one applies this procedure to a linear function, then the result is again linear, since the exponential of a constant is a constant. This suggests that linear MOMA problems should correspond to SOMA problems of the form (\ref{moma}) with linear reward functions. Such linear SOMA problems are known as \emph{optimal transport} problems. In this section, we confirm the conjecture that there is a one-to-one relationship between solutions of the discrete optimal transport problem on the one hand, and solutions of related linear MOMA problems on the other hand.

\begin{remark}
For MOMA minimization problems with increasing and strictly convex disutility functions, a result that is analogous to the one in \cite{Gale79} can be proven. The relation between the cost functions $g_{ij}(x)$ in the MOMA minimization problem and the cost functions $f_{ij}(x)$ in the SOMA minimization problem is given by the same formula (\ref{momasoma}) as in the maximization case.
\end{remark}

\begin{remark}
In the single-objective case, maximization and minimization problems can be simply converted into each other by the usual device of changing the sign of the objective function. The same is not true for MOMA problems in which we are looking for balanced solutions, since here we require that objective functions are increasing, both in the case of maximization and in the case of minimization. Instead, one can make use of the connection via the equivalent SOMA problem as specified in (\ref{momasoma}). It is seen that a matrix $X$ is a balanced solution of the MOMA maximization problem with reward functions $g_{ij}$ if and only if $X$ is a balanced solution of the MOMA minimization problem with cost functions $\hat{g}_{ij}$ (and the same row and column constraints), when the functions $g_{ij}$ and $\hat{g}_{ij}$ are related via
\begin{equation} \label{conjug}
\hat{g}'_{ij}(x) = \frac{1}{g'_{ij}(x)}\,.
\end{equation}
The conjugation defined by (\ref{conjug}) relates for instance $\log x$ (increasing, concave) to $\frac{1}{2}x^2$ (increasing, convex). This means that a logarithmic MOMA maximization problem is equivalent to a quadratic MOMA minimization problem. For another example, the increasing and concave power functions $x^{1-\gamma}/(1-\gamma)$ with $\gamma > 0$, popular in utility theory, are linked by (\ref{conjug}) to the increasing and convex power functions $x^{1+\gamma}/(1+\gamma)$. A linear MOMA maximization problem with coefficients $b_{ij}$ is equivalent to a linear MOMA minimization problem with coefficients $b_{ij}^{-1}$.
\end{remark}

The optimal transport problems that will be considered below, in the maximization case, are of the discrete form
\begin{subequations} \label{ot}
\begin{align}
\sum_{i=1}^n \sum_{j=1}^m a_{ij}& x_{ij} \rightarrow \text{ max} \\
\text{subject to} \quad \sum_{i=1}^n x_{ij} & = c_j \quad (j = 1,\dots,m) \\
\sum_{j=1}^m x_{ij} & = r_i \quad (i = 1,\dots,n) \\
x_{ij} & \geq 0.
\end{align}
\end{subequations}
The given row sums $r_i$ and column sums $c_j$ are all supposed to be positive, and the global feasibility condition $\sum_i r_i = \sum_j c_j$ is supposed to be satisfied. Under these conditions, the feasible set is not empty; for instance, we can take $x_{ij} = r_ic_j/s$ where $s:= \sum_i r_i = \sum_j c_j$.
The problem (\ref{ot}) is known as the Monge-Kantorovich problem. It is a linear programming problem; its dual is given by
\begin{subequations}
\begin{align}
\sum_{i=1}^n &\, \lambda_i r_i + \sum_{j=1}^m \mu_j c_j \rightarrow \text{ min} \\
\text{subject to} \quad & \lambda_i + \mu_j \geq a_{ij} \quad (i = 1,\dots,n; \; j = 1,\dots,m).
\end{align}
\end{subequations}
\begin{theorem}
A matrix $X$ is a balanced solution of the linear \ed{MOMA} problem \ed{(\ref{linmoma})}, taken as a maximization problem, if and only if $X$ is an optimal solution of the optimal transport problem \ed{(\ref{ot})} with the same row and column constraints, and
\begin{equation} \label{log}
a_{ij} = \log b_{ij}.
\end{equation}
\end{theorem}

\begin{proof}
By linear programming duality, which in this case is Monge-Kantorovich duality (see for instance \cite[Thm.\,2.2]{Galichon}), a matrix $X \in \mR^{n \times m}_+$ is a solution of the optimal transport problem (\ref{ot}) if and only if there are $\lambda_i \in \mR$ ($i = 1,\dots,n$) and $\mu_j \in \mR$ ($j=1,\dots,m$) such that the following conditions hold:
\begin{multline}
\big(x_{ij}=0 \text{ and } a_{ij} \leq \lambda_i + \mu_j\big)
\text{ or } \big(x_{ij} \geq 0 \text{ and } a_{ij} = \lambda_i + \mu_j\big) \qquad (i=1,\dots,n;\; j = 1,\dots,m), \\
\sum_{i=1}^n x_{ij} = c_j \quad (j=1,\dots,m), \qquad
\sum_{j=1}^m x_{ij} = r_i \quad (i=1,\dots,n). \hspace{2cm}
\end{multline}
Under (\ref{log}), these conditions are identical to the conditions (\ref{char}), since we can take $\alpha_i = \exp(-\lambda_i)$ and $\beta_j = \exp \mu_j$. The statement of the theorem follows.
\qed
\end{proof}

The value of the function $\sum_i \sum_j (\log b_{ij}) x_{ij}$ for general $x_{ij}$ cannot be determined on the basis of information concerning only the values of the functions $\sum_j b_{ij}x_{ij}$. This shows that, even in the linear case, the relation between MOMA problems and SOMA problems as reflected in the theorem above cannot be viewed as a scalarization procedure. The term ``social welfare function'' used by
Gale and Sobel \cite{Gale79} must be interpreted carefully; it should be noted that in the single-objective problem, the welfare is served of agents who are different from those that appear in the corresponding multi-objective problem.

\begin{remark}
The relation (\ref{log}) can be used to translate properties from linear SOMA problems to linear MOMA problems and vice versa. For instance, it is well known that the optimal transport problem can be solved very easily by a greedy algorithm when the weight matrix satisfies the \emph{Monge property}, which requires that $a_{i_1 j_1} + a_{i_2 j_2} \geq a_{i_1 j_2} + a_{i_2 j_1}$ for all $1 \leq i_1 < i_2 \leq n$ and $1 \leq j_1 < j_2 \leq m$; see \cite{Burkard} for an extensive discussion.\footnote{The direction of the inequality is reversed here with respect to the convention in the cited paper, because we consider a maximization problem, rather than a minimization problem as in \cite{Burkard}.} The corresponding condition for the related multi-objective problem is that the determinants of all 2-by-2 submatrices of the coefficient matrix $(b_{ij})$ should be non-negative.
\end{remark}

\section{Iterative solution method} \label{iterative}

\subsection{Direct iteration}

For strictly concave single-objective matrix allocation problems of the form (\ref{soma}), there is a natural iterative solution method, which may be viewed as an application of the alternating projections method due to Bregman \cite{Bregman}. The idea of solving matrix allocation problems by an iterative process that switches between row and column constraints already occurs in the early contributions of Kruithof \cite{Kruithof} and Deming and Stephan \cite{Deming}.
The algorithm proposed in \cite{Kruithof} and \cite{Deming} is known as the Iterative Proportional Fitting Procedure (IPFP), or also as the RAS method. The iterative solution method for SOMA maximization problems that we describe below is a direct generalization of IPFP. For simplicity, it is assumed that the reward functions $f_{ij}(x)$ in (\ref{soma}) satisfy $\lim_{x \downarrow 0} f'_{ij}(x) = \infty$ and $\lim_{x \rightarrow \infty} f_{ij}(x) = -\infty$; the corresponding conditions for the MOMA reward functions related via (\ref{momasoma}) are usually referred to as Inada conditions. These conditions imply, in particular, that the non-negativity constraints are redundant (optimal solutions must be strictly positive).

The standard Lagrange optimization method for the problem (\ref{soma}) gives rise to the equations
\begin{equation}
f_{ij}'(x_{ij}) = \lambda_i + \mu_j \qquad (i=1,\dots,n;\; j = 1,\dots,m)
\end{equation}
where $\lambda_i$ ($i=1,\dots,n$) and $\mu_j$ ($j=1,\dots,m$) are Lagrange multipliers corresponding to row and column constraints respectively. Under the strict concavity assumptions on the functions $f_{ij}$, the derivatives $f_{ij}'$ are strictly decreasing. Consequently, these functions have inverses, which will be denoted by $F_{ij}$. Under the Inada conditions, the inverse functions $F_{ij}$ are defined on all of the real line. We can then rewrite the row and column constraints in terms of the multipliers $\lambda_i$ and $\mu_j$:
\begin{subequations}
\begin{align}
\sum_{i=1}^n F_{ij}(\lambda_i + \mu_j) & = c_j \quad (j = 1,\dots,m) \label{iter1} \\
\sum_{j=1}^m F_{ij}(\lambda_i + \mu_j) & = r_i \quad (i = 1,\dots,n) \label{iter2}
\end{align}
\end{subequations}
Starting with an initial guess for the values of $\lambda_i$, one can use the equations (\ref{iter1}) to compute corresponding multipliers $\mu_j$. Each multiplier $\mu_j$ appears in exactly one of the equations. Depending on the nature of the functions $F_{ij}$, it may be possible to write down the answer in analytic form, but even if this is not the case, the problem of finding $\mu_j$ only calls for determining the root of a scalar monotone function. When values for $\mu_j$ have been found in this way, one can use the equations (\ref{iter2}) to find updated values of the multipliers $\lambda_i$. Again the equation system is diagonal in the sense that each $\lambda_i$ occurs in exactly one of the equations. After solving for $\lambda_i$ ($i=1,\dots,n$), one can determine updates for $\mu_j$ ($j = 1,\dots,m$), and so on.

Convergence of the iterative proportional fitting procedure and of its generalization to infinite dimensions has been studied extensively; see for instance \cite{Brown,Sinkhorn,Ireland,Fienberg,Csiszar,Rueschendorf,Bauschke}. Motivated by a MOMA problem in risk sharing, Pazdera et al.\ \cite{PSW1} proved convergence of the iterative procedure for a broad class of strictly concave utilities, on the basis of a nonlinear version of the Perron-Frobenius theorem due to \cite{Oshime}.
One may write down an analogous iterative procedure for optimal transport problems, or equivalently for their multi-objective counterparts. While in this case the objective functions are linear rather than strictly concave, the iterative procedure can still be followed; however, as will be seen below, the procedure is not effective as a solution method. Nevertheless, the iteration in the linear case is worked out below, because it sheds light on the behavior of regularized versions.

We follow the MOMA viewpoint, which leads naturally to a multiplicative formulation. Iteration takes place between the utility weights $\alpha_i$ and the multipliers $\beta_j$. Starting with positive initial values for the weights $\alpha_i$ (for example, one might choose $\alpha_i = 1$ for all $i$), the multipliers $\beta_j$ are to be found from the conditions for optimality of the weighted-sum problem (\ref{momacomp}) and the column constraints. It is easily verified that, given a vector of weights $\alpha$, the vector of multipliers $\beta$ is such that (\ref{momacomp}) as well as the column constraints are satisfied for some $x_{ij}$ if and only if
$\beta_j = \max_i \alpha_i b_{ij}$ for $j=1,\dots,m$.
Likewise, given a vector of multipliers $\beta$, the vector of weights $\alpha$ is such that (\ref{momacomp}) as well as the row constraints are satisfied for some $x_{ij}$ if and only if
$\alpha_i = \min_j \beta_j/b_{ij}$ for $i=1,\dots,n$.
This leads to the iteration $\alpha\kn = \pOT(\alpha\ki)$, where the mapping $\pOT$ from $\mR^n_+$ into itself is defined by $\pOT(\alpha) = \hat{\alpha}$ with
\begin{equation} \label{pot}
\beta_j = \max_i \alpha_i b_{ij} \quad (j=1,\dots,m), \qquad
\hat{\alpha}_i = \min_j \beta_j/b_{ij} \quad (i=1,\dots,n).
\end{equation}
\vskip2mm\noindent
The behavior of this iterative procedure may be understood as follows. The update $\alpha_i^{(k+1)}$ is such that
$$
\min_j \frac{\max_i \alpha_i^{(k)} b_{ij}}{\alpha_i^{(k+1)}b_{ij}} = 1.
$$
In other words, the new vector of weights $\alpha_i^{(k+1)}$ is determined such that
$
\alpha_i^{(k+1)}b_{ij} \leq \max_i \alpha_i^{(k)} b_{ij}
$
for all $j$, with equality for at least one value of $j$. The quantity $\max_i \alpha_i^{(k)} b_{ij}$ may be called the \emph{column maximum} for column $j$ in the row-scaled matrix with elements $\alpha_i^{(k)} b_{ij}$. If row $i$ in this matrix does not contain an element whose value is equal to the column maximum, then the row will receive a renewed scaling by multiplication by $\alpha_i^{(k+1)}$, so that the column maximum will be achieved. The renewed scaling is the minimal one which has this effect, and therefore it does not affect the value of any of the column maxima. After one iteration, each row in the row-scaled matrix contains a column maximal element, and the iteration arrives at a fixed point.

While convergence therefore takes place after at most one step, there are many fixed points. Any row scaling which is such that every row contains a column maximal element corresponds to a fixed point of the iteration. Moreover, the decision matrix cannot be written as a function of the weights $\alpha_i$ and the multipliers $\beta_j$, in contrast to the strictly concave case.

\subsection{Regularization}

A possible solution to remedy the unsatisfactory behavior of the direct iterative algorithm for the optimal transport problem as discussed above is to apply \emph{regularization}. Here, we consider an additive regularization of the optimal transport problem, and we aim at finding the corresponding regularization in the multi-objective framework. In order to regularize, one can replace the linear reward functions $a_{ij}x_{ij}$ by concavified versions
\begin{equation} \label{somareg}
f_{ij}(x_{ij},\eta) = a_{ij}x_{ij} + \eta f_1(x_{ij})
\end{equation}
where the regularization parameter (or ``temperature'') $\eta$ is a small positive number, and the regularization function $f_1(x)$ is strictly concave. The optimization problem defined by these regularized reward functions is strictly concave, and the relation (\ref{momasoma}) can be used to determine the corresponding multi-objective problem. Defining the function $g_1(x)$ (up to an irrelevant constant) from the
regularization function $f_1(x)$ by $\log g_1'(x) = f_1'(x)$ as in (\ref{momasoma}), one finds that the MOMA regularization corresponding to (\ref{momareg}) is given by
\begin{equation} \label{momareg}
g_{ij}(x_{ij},\eta) = b_{ij} g_\eta(x_{ij})
\end{equation}
where $b_{ij} = \exp(a_{ij})$ and the function $g_\eta(x)$ satisfies
\begin{equation} \label{momapar}
\log g_\eta'(x) = \eta \log g_1'(x).
\end{equation}
The relation (\ref{momapar}) has a straightforward interpretation in terms of the Arrow-Pratt coefficient of risk aversion. The Arrow-Pratt risk aversion function $r_\eta(x)$ associated to the utility function $g_\eta(x)$ is given by $r_\eta(x) = - g''_\eta(x)/g'_\eta(x)$. With this definition, the relation (\ref{momapar}) is equivalent (up to an irrelevant constant) to
\begin{equation} \label{ra}
r_\eta(x) = \eta r_1(x).
\end{equation}
This shows a rather natural relation between the risk aversion functions corresponding to $g_\eta$ and $g_1$.

A regularization function that is frequently chosen for the optimal transport problem is the entropic function given by
\begin{equation}
f_1(x) = - x\log x
\end{equation}
\cite{Cuturi,Benamou,Peyre,Galichon}. Up to a positive factor, which we can ignore on the basis of Remark \ref{inv}, the corresponding regularization of the MOMA problem is $g_1(x) = 1/x$. From (\ref{momapar}), it follows that the MOMA regularization in this case can be written as
\begin{equation} \label{aij}
g_{ij}(x_{ij};\eta) = b_{ij} \, \frac{x_{ij}^{1-\eta}}{1-\eta}\,.
\end{equation}
The function that appears here is known as \emph{power utility}, or also as \emph{isoelastic utility} or \emph{CRRA utility} (constant relative risk aversion). It is seen, therefore, that \emph{entropic} regularization of the optimal transport problem corresponds to \emph{isoelastic} regularization of the corresponding multi-objective matrix allocation problem.

In the case of entropic regularization of optimal transport problems, or, equivalently, isoelastic regularizations of linear MOMA problems, the iterative algorithm based on (\ref{iter1}--\ref{iter2}) specializes as follows. Working from the multiplicative form, the $k$-th iterate $x\ki_{ij}$ is determined in terms of the corresponding iterates of the multiplicative parameters by\footnote{Usage of the notation $\alpha$ and $\beta$ here is in line with the notation in Section \ref{character}.}
$$
\alpha_i\ki g'_{ij} \big(x\ki_{ij}\big) = \beta_j\ki.
$$
In the special case (\ref{aij}), this produces
$$
x_{ij}\ki = \Bigg( \frac{\alpha_i\ki b_{ij}}{\beta_j\ki} \Bigg)^{1/\eta}.
$$
Applying the column and row constraints successively leads to the iteration mapping $\preg$ from $\mR^n_+ \setminus \{0\}$ into $\mR^n_{++}$ defined by $\preg(\alpha) = \hat{\alpha}$, where the vector $\hat{\alpha}$ is obtained from
\begin{equation} \label{updates}
\beta_j = \frac{\| \alpha_\cdot b_{\cdot j} \|_{1/\eta}}{c_j^\eta} \quad (j=1,\dots,m), \qquad
\hat{\alpha}_i = \frac{r_i^\eta}{\| b_{i\cdot}/\beta_\cdot \|_{1/\eta}} \quad (i=1,\dots,n).
\end{equation}
Here, $\| \cdot \|_p$ (with $p = 1/\eta$) refers to the usual vector $p$-norm $\|x\|_p = \big(\sum_{i=1}^n |x_i|^p\big)^{1/p}$. The notations $\alpha_\cdot b_{\cdot j}$ and $b_{i\cdot}/\beta_{\cdot}$ are used to indicate the $n$-vector with entries $\alpha_i b_{ij}$ ($i=1,\dots,n$) and the $m$-vector with entries $b_{ij}/\beta_j$ ($j=1,\dots,m$) respectively. In the limit, as the regularization parameter $\eta$ tends to 0, the update mappings that are defined here converge to the ones in (\ref{pot}).

\begin{remark}
The update equations in (\ref{updates}) can equivalently be written as follows:
\begin{equation} \label{benamou}
\beta_j^{1/\eta} = \frac{\sum_i (\alpha_i b_{ij})^{1/\eta}}{c_j}\,, \qquad
\hat{\alpha}_i^{1/\eta} = \frac{r_i}{\sum_j (b_{ij}/\beta_j)^{1/\eta}}\,.
\end{equation}
Instead of computing the sequences $\alpha_i^{(k)}$ and $\beta_j^{(k)}$, one can therefore also process the iteration with the sequences $\tilde{\alpha}_i^{(k)}:=\big(\alpha_i^{(k)}\big)^{1/\eta}$ and $\tilde{\beta}_j^{(k)}:=\big(\beta_j^{(k)}\big)^{1/\eta}$. This alternative form of the iteration algorithm is the one presented by Benamou et al.\ \cite{Benamou}. It coincides with the IPFP algorithm, applied to the matrix with entries $b_{ij}^{1/\eta}= \exp(a_{ij}/\eta)$. The advantages and disadvantages of various possible formulations are discussed below in Section \ref{numerics}.
\end{remark}

\begin{remark} \label{deming}
As noted above, the IPFP algorithm corresponds to a SOMA minimization problem in which the cost function is given by
relative entropy. It is readily verified, via the differentiation-exponentiation-integration route (\ref{momasoma}), that
the related MOMA minimization problems refers to agents whose utility functions are \emph{quadratic}. There is some historical
interest to this observation. The IPFP method was presented in 1940 by Deming and Stephan \cite{Deming} as an algorithm that achieves a best approximation in the quadratic sense to a given matrix, subject to the row and column constraints. The authors soon discovered, however,
that their claim was mistaken \cite{Stephan}. Decades later, it was established by Ireland and Kullback \cite{Ireland} that the approximation computed via IPFP is optimal in the sense of Kullback-Leibler divergence (relative entropy), rather than in the quadratic sense. The error in \cite{Deming} may be softened a bit by the observation that, in the MOMA interpretation, IPFP does solve a quadratic minimization problem.
\end{remark}

\subsection{Perron-Frobenius analysis}

The mapping (\ref{updates}), when extended to take the value 0 at 0, is a homogeneous mapping of the nonnegative cone into itself. The convergence behavior of iterations defined by such mappings has been studied extensively, both in finite-dimensional and in infinite-dimensional spaces; see for instance
\cite{SinkhornKnopp,Franklin,Borwein,Lemmens,Georgiou}. A key tool in this literature is the Hilbert metric, which is defined for positive vectors of equal length by
\begin{equation}
d(x,y) = \log \frac{\max_i x_i/y_i}{\min_i x_i/y_i}\,.
\end{equation}
The Hilbert metric is in fact only a pseudo-metric on the positive cone, since points of the same ray have distance zero, but on the open
unit simplex it is a true metric. The use of this metric for the study of eigenvectors of positive linear mappings was initiated in
\cite{Birkhoff}, while its application to nonlinear mappings starts with \cite{Bushell}. The purpose of this section is
to show how a result of Oshime \cite{Oshime} in nonlinear Perron-Frobenius theory can be used to prove with little effort that the iteration (\ref{updates}) converges to a unique (up to a positive multiplicative factor) fixed point. This follows the development in \cite{PSW1}, but the proof can be much shorter in the present case. The approach facilitates comparison with the ineffective iteration (\ref{pot}).

The notion of nonsectionality is needed, which may be thought of as a nonlinear version of the well known irreducibility condition of linear Perron-Frobenius theory. It is defined as follows. In the definition below, the notation $x_R$ where $x$ is an $n$-vector and $R$ is a subset of the index set $\{1,\dots,n\}$ denotes the subvector with entries $x_i$, $i \in R$.

\begin{definition} \label{nonsect} \cite{Oshime}
A mapping $\phi$ of $\mR^n_+$ into itself is \emph{nonsectional} if, for
every decomposition of the index set $\{1,\dots,n\}$ into two complementary nonempty
subsets $R$ and $S$, there exists $s \in S$ such that
\begin{itemize}
\item[(i)] $(\phi(x))_s > (\phi(y))_s$ for all $x,\,y \in \mR^n_+$ such that
$x_R > y_R$ and $x_S=y_S>0$;
\item[(ii)] $(\phi(x^k))_s \rightarrow \infty$ for all sequences $(x^k)_{k=1,2,\dots}$
in $\mR^n_+$ such that $x^k_R \rightarrow \infty$, while $x^k_S$ is fixed and positive.
\end{itemize}
\end{definition}

The following theorem, as cited in \cite{PSW1}, can be used to prove uniqueness of the fixed points of the iteration defined by (\ref{updates}).

\begin{theorem} {\rm \cite[Thm.\,8, Remark 2]{Oshime}} \label{Oshime}
If a mapping $\phi$ from $\mR^n_+$ into itself is continuous, monotone,
homogeneous, and nonsectional, then the mapping $\phi$ has a
positive eigenvector, which is unique up to scalar multiplication, with a positive associated
eigenvalue. In other words, there exist
a constant $\theta^* > 0$ and a vector $x^* \in \mR^n_{++}$ such that $\phi(x^*)=
\theta^*x^*$, and if $\theta >0$ and $x \in \mR^n_{++}$ are such that $\phi(x)
=\theta x$, then $x$ is a scalar multiple of $x^*$.
\end{theorem}

The existence of a unique positive eigenvector to the mapping defined by (\ref{updates}) does not immediately imply the existence of a fixed point; it needs to be shown that the corresponding eigenvalue is 1. This follows from the global feasibility constraint (cf.\ similar arguments in \cite[p.\,246]{Menon} and \cite{PSW1}).

\begin{lemma} \label{lemma1}
If, in the iteration mapping \ed{(\ref{updates})}, we have $\hat{\alpha} = \theta \alpha$ for $\alpha \in \mR^n\setminus\{0\}$, then $\theta = 1$.
\end{lemma}

\begin{proof}
The equations that define $\hat{\alpha}$ from $\beta$ and $\beta$ from $\alpha$, written in implicit form, are
$$
\sum_{j=1}^m \Big( \frac{\hat{\alpha}_i b_{ij}}{\beta_j} \Big)^{1/\eta} = r_i\,, \quad
\sum_{i=1}^n \Big( \frac{\alpha_i b_{ij}}{\beta_j} \Big)^{1/\eta} = c_j.
$$
The global feasibility constraint $\sum_i r_i = \sum_j c_j$ therefore implies that
$$
\sum_{i=1}^n \sum_{j=1}^m \Big( \frac{\hat{\alpha}_i b_{ij}}{\beta_j} \Big)^{1/\eta} = \sum_{i=1}^n \sum_{j=1}^m \Big( \frac{\alpha_i b_{ij}}{\beta_j} \Big)^{1/\eta}.
$$
Consequently, if $\hat{\alpha} = \theta \alpha$, then $\theta$ must be equal to 1.
\qed
\end{proof}

We can verify the conditions in Thm.\,\ref{Oshime} for the mapping defined in (\ref{updates}), and compare them to the properties of the update mapping defined by (\ref{pot}). A mapping $\phi$ from (a subset of) $\mR^{n_1}$ to $\mR^{n_2}$ is said to be \emph{strongly monotone} if the conditions $x_i \geq y_i$ for $i = 1,\dots,n_1$ and $x_i > y_i$ for some $i \in \{1,\dots,n_1\}$ imply $(\phi(x))_i > (\phi(y))_i$ for all $i = 1,\dots,n_2$.

\begin{proposition} \label{lemma2}
The mapping from the non-negative cone $\mR_+$ to itself defined by \ed{(\ref{pot})} is continuous, monotone, and homogeneous. The mapping \ed{(\ref{updates})} from $\mR_+$ into itself is continuous, strongly monotone, homogeneous, and nonsectional.
\end{proposition}

\begin{proof}
Continuity, monotonicity, and homogeneity are immediate from the definitions, both for the non-regularized mapping and for the regularized version. To show strong monotonicity in the case of (\ref{updates}), note that, due to the assumed positiveness of all coefficients $b_{ij}$, an increase of any of the components of the weight vector $\alpha$ implies an increase of all of the components of the vector $\beta$, which in its turn implies an increase of all components of the image vector $\hat{\alpha}$. Strong monotonicity implies that condition (i) in the definition of nonsectionality is satisfied. Moreover, when one of the components of $\alpha$ tends to infinity, the same is true for all of the components of $\beta$, and consequently for all of the components of $\hat{\alpha}$. This shows that condition (ii) is satisfied by the mapping (\ref{updates}).
\qed
\end{proof}

\begin{corollary}
The mapping $\phi_{\eta}$ defined by \ed{(\ref{updates})} has a unique (up to multiplication by a positive constant) fixed point in $\mR^n_{++}$. For any initial point $\alpha^{(0)}$ in $\mR^n_+\setminus\{0\}$, the sequence of normalized iterates defined by $\alpha\kn = \preg(\alpha\ki)/\|\preg(\alpha\ki)\|_1$ converges to the unique fixed point on the open unit simplex.
\end{corollary}

\begin{proof}
The first claim follows from Prop.\,\ref{lemma2} and Thm.\,\ref{Oshime}. By a basic result in Perron-Frobenius theory (see for instance \cite[Ch.\,2]{Lemmens} and \cite[Lemma 4.4]{PSW1}), a homogeneous and strongly monotone mapping from the positive cone into itself is contractive in the sense of the Hilbert metric. Given the uniqueness of the fixed point, convergence of the iteration then follows from a result of Nadler \cite{Nadler} (see also \cite[Cor.\,5.4]{PSW1}).
\qed
\end{proof}

The non-regularized mapping (\ref{pot}) does not satisfy the property of nonsectionality. Indeed, otherwise uniqueness of a fixed point would follow as in the corollary above, whereas it has already been argued that the mapping (\ref{pot}) has many fixed points, which are not related to each other by multiplication by a positive constant. To see the violation of nonsectionality more directly, consider a point $\alpha \in \mR_{++}^n$ in which the first component $\alpha_1$ is large relative to the other components, so that, for all $j$, $\alpha_1 b_{1j} > \alpha_i b_{ij}$ for $i=2,\dots,n$. A sufficiently small increase of the components $\alpha_2,\dots,\alpha_n$ will then have no effect on any of the parameters $\beta_j$ defined in (\ref{pot}), and consequently there is no effect on the update $\hat{\alpha}$ either. This means that condition (i) in the definition of nonsectionality is not satisfied with $S=\{1\}$ and $R=\{2,\dots,n\}$. It also follows that the mapping defined by (\ref{pot}) is not strongly monotone.

\section{Numerics} \label{numerics}

Alternative implementations of the regularization method for optimal transport problems can be derived from (\ref{updates})
and from (\ref{benamou}). These formulas are mathematically equivalent, but the corresponding implementations may differ
in their numerical properties, especially when the regularization parameter is close to 0.
As noted in \cite{Benamou}, when the parameter $\eta$ is very small, the computation of a number of the form $b^{1/\eta}$ may lead to numerical overflow or underflow. Such problems are less likely in the implementation (\ref{updates}), which uses the operations of raising a number to the power $\eta$ and computing the $p$-norm of a vector with $p=1/\eta$; both of these operations can be implemented without the creation of intermediate results that are very large or very small.\footnote{The size of the machine epsilon still imposes a bound on how small the parameter $\eta$ can be made. As soon as the $1/\eta$-norms that are computed become numerically equal to $\infty$-norms, and numbers raised to the power $\eta$ become numerically equal to 1, the algorithm (\ref{updates}) degenerates into the ineffective non-regularized version (\ref{pot}).}

There are different forms of the IPFP algorithm. It can be written in a vector iteration form, starting from a given positive initial matrix $(x_{ij}^{(0)})$ and a positive initial vector $u^{(0)}$:
\begin{equation} \label{vec}
v_j^{(k+1)} = \frac{c_j}{\sum_i u_i^{(k)} x_{ij}^{(0)}}\,, \qquad
u_i^{(k+1)} = \frac{r_i}{\sum_j x_{ij}^{(0)} v_j^{(k+1)}}\,, \qquad
x_{ij}^{(k+1)} = u_i^{(k+1)}x_{ij}^{(0)}v_j^{(k+1)}.
\end{equation}
Alternatively, a matrix iteration can be used:
\begin{equation} \label{mat}
x_{ij}^{(k+\frac{1}{2})} = x_{ij}^{(k)} \frac{c_j}{\sum_i x_{ij}^{(k)}}\,, \qquad
x_{ij}^{(k+1)} = \frac{r_i}{\sum_j x_{ij}^{(k+\frac{1}{2})}} \, x_{ij}^{(k+\frac{1}{2})}.
\end{equation}
It can be verified immediately that, if the initial vector $u^{(0)}$ is chosen such that $u_i^{(0)} = 1$ for all $i$, the two recursions
produce the same sequence of matrices $x_{ij}^{(k)}$. One might say that (\ref{vec}) presents a cumulative form of the computation, whereas (\ref{mat}) is the corresponding incremental form. Corresponding to the vector recursion (\ref{updates}), a matrix iteration can
analogously be defined as follows. Starting from a given matrix $(z_{ij}^{(0)})$, define
\begin{equation} \label{upd}
z_{ij}^{(k+\frac{1}{2})} =  s_j^{(k)} z_{ij}^{(k)}, \qquad
z_{ij}^{(k+1)} = t_i^{(k)} z_{ij}^{(k+\frac{1}{2})}
\end{equation}
where the incremental multipliers $s^{(k)} \in \mR^m$ and $t^{(k)} \in \mR^n$ are given by
\begin{equation} \label{incr}
s_j^{(k)} =  \frac{c_j^\eta}{\| z^{(k)}_{\cdot j} \|_{1/\eta}}\,, \qquad
t_i^{(k)} =  \frac{r_i^\eta}{ \| z^{(k+\frac{1}{2})}_{i\cdot} \|_{1/\eta} } \,.
\end{equation}
The matrices $(z_{ij}^{(k)})$ relate to the matrices $(x_{ij}^{(k)})$ as defined in (\ref{updates}) via $x_{ij} = z_{ij}^{1/\eta}$.
Starting from the initialization $z_{ij} = b_{ij}$, one can therefore iterate the matrices $(z_{ij})$, and extract the
approximate solution of the optimal transport problem $(x_{ij})$ at the end.

Extracting the approximate solution only at the end raises the question how to monitor the quality of the solution in intermediate stages, in order to obtain a stopping criterion for the iteration. By construction of the algorithm, after each full step,
the row sums of the matrix $(x_{ij}^{(k)})$ are equal to the prescribed row sums $r_i$, up to machine precision. Therefore, it
is natural to measure the accuracy of the solution by the deviation of the column sums from the prescribed
values $c_j$. From the equality $\sum_{j=1}^m x_{ij} = r_i$ for all $i$, it follows that
$\sum_{j=1}^m \sum_{i=1}^n x_{ij} = \sum_{j=1}^m c_j$.
Consequently, the Hilbert metric can be used as a valid distance measure between the vector of column sums
of the approximate solution and the vector $c$. We have
\begin{align}
d\Big(\sum_{i=1}^n x^{(k)}_{i\cdot},c \Big) & = \log \frac{\max_j \big(\sum_i x^{(k)}_{ij}\big)/c_j}{\min_j \big(\sum_i x^{(k)}_{ij}\big)/c_j} = \log \frac{\max_j \big(\| z^{(k)}_{\cdot j} \|_{1/\eta} /c_j^{\eta}\big)^{1/\eta}}{\min_j \big(\| z^{(k)}_{\cdot j} \|_{1/\eta} /c_j^{\eta}\big)^{1/\eta}} \nonumber \\
& = \frac{1}{\eta} \log \frac{\max_j \| z^{(k)}_{\cdot j} \|_{1/\eta} /c_j^{\eta}}{\min_j \| z^{(k)}_{\cdot j} \|_{1/\eta}/ c_j^{\eta}} = \frac{1}{\eta} \log \frac{\max_j s_j^{(k)}}{\min_j s_j^{(k)}} \label{critval} = \frac{1}{\eta} \, d(s^{(k)},\one)
\end{align}
where $s_j^{(k)}$ is defined in (\ref{incr}), and $\one$ is the all-ones vector. Computation of the Hilbert metric is therefore an easy byproduct of the updating rules (\ref{upd}--\ref{incr}), and one can monitor the progress of the convergence without actually computing the intermediate approximate solutions $(x_{ij}^{(k)})$.

As noted above, the non-regularized vector iteration (\ref{pot}) can have at most one nontrivial step. In this step,
it produces a row balancing ensuring that every row of the weight matrix $(b_{ij})$ contains at least one element that is
maximal among all elements in the column that it belongs to. From a numerical perspective, this type of row equilibration is
possibly useful.

An iterative algorithm to solve the regularized version of the optimal transport problem
may now be written as follows. Inputs to the computation are: the weight matrix $(a_{ij})$; the column sums $c_j$;
the row sums $r_i$; the chosen regularization parameter $\eta$; and, finally, the tolerance level that is chosen as a
stopping criterion for the iteration.
\begin{enumerate}
\item Define $b_{ij} = \exp(a_{ij})$.
\item Define $\hat{b}_{ij} = \alpha_i b_{ij}$, with $\alpha_i = \min_j\!\big((\max_i b_{ij})/b_{ij}\big)$
(preliminary row equilibration).
\item Initialize $\big(z_{ij}^{(0)}\big) = (\hat{b}_{ij})$.
\item Carry out the iteration (\ref{upd}--\ref{incr}), until the quantity
$\frac{1}{\eta} \log \!\big((\max_j s_j^{(k)})/(\min_j s_j^{(k)})\big)$
is less than the chosen tolerance level.
\item Compute the approximate solution by $x_{ij} = \big(z^{(k)}_{ij}\big)^{1/\eta}$, where $k$ is the index of the final iteration step.
\end{enumerate}

Numerical results for a test problem are shown in the figures below. The weight matrix is defined by a discretization, on
a grid of size $256 \times 256$, of the function $a(x,y) = \sin\!\big(4\pi \big( (x-0.5)^2 + (y-0.5)^2 \big)\big)$ defined on the
unit square.
Marginals are given by scaled discretizations of the functions $c(x) = |x-0.5|$ and $r(y) = |y-0.5|$. Fig.\,\ref{problemdata}
gives a visual description of the problem data (darker cells correspond to higher values of $a_{ij}$). The results of computations are shown in Fig.\,\ref{results} for $\eta = 10^{-2}$, $\eta=10^{-3}$, and $\eta=10^{-4}$. The tolerance level is set at $0.01$ in each case. CPU time per iteration step is mainly determined by the grid size, and is not
influenced much by the choice of the regularization parameter $\eta$. However, the number of
iteration steps required to reach a given level of convergence appears to be approximately inversely proportional to $\eta$.
The computation time is therefore also inversely proportional to $\eta$.

\begin{figure}
\begin{center}
\begin{subfigure}[b]{.4\linewidth}
\includegraphics[width=6cm]{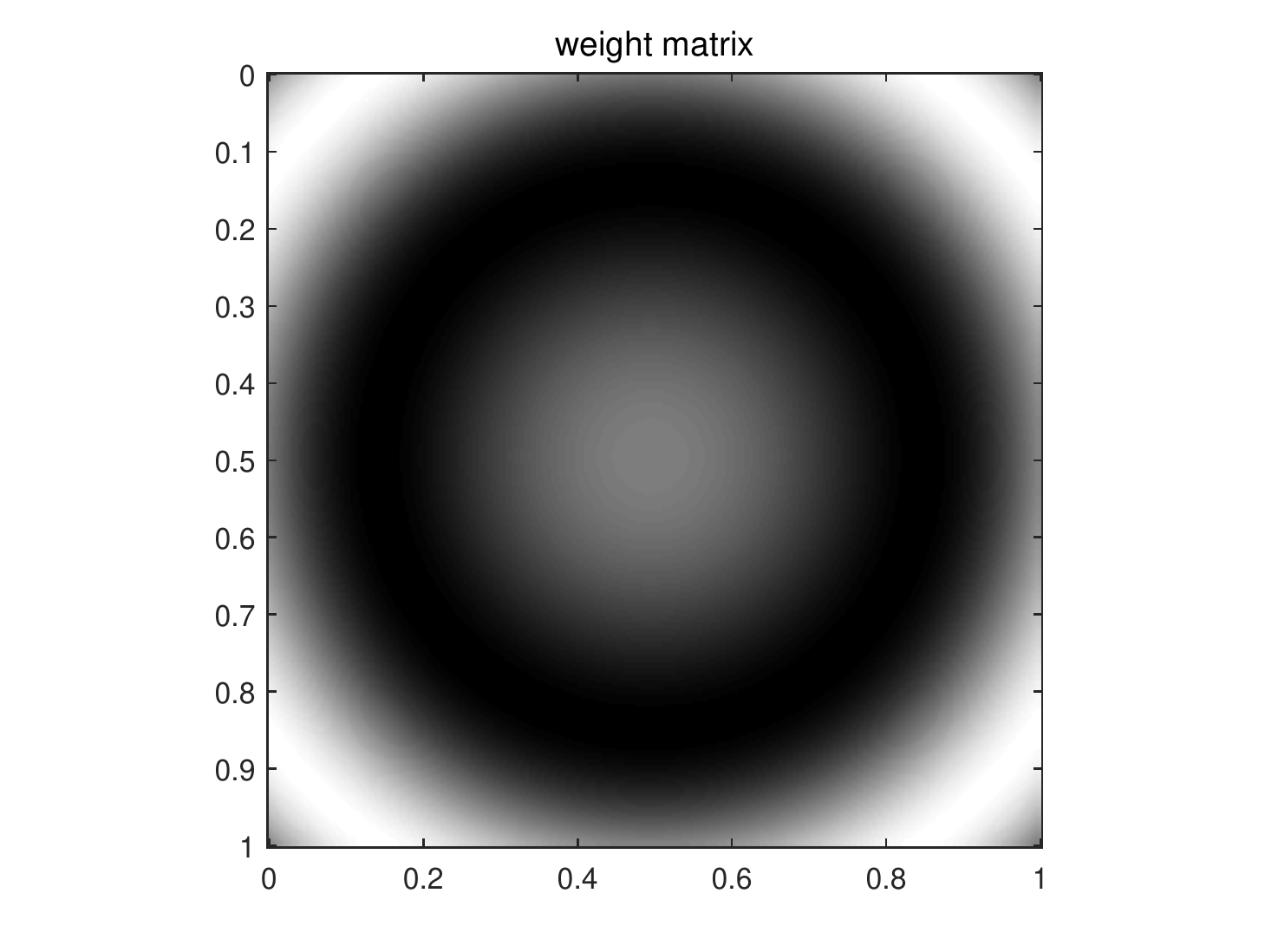}
\end{subfigure}
\begin{subfigure}[b]{.4\linewidth}
\includegraphics[width=6cm]{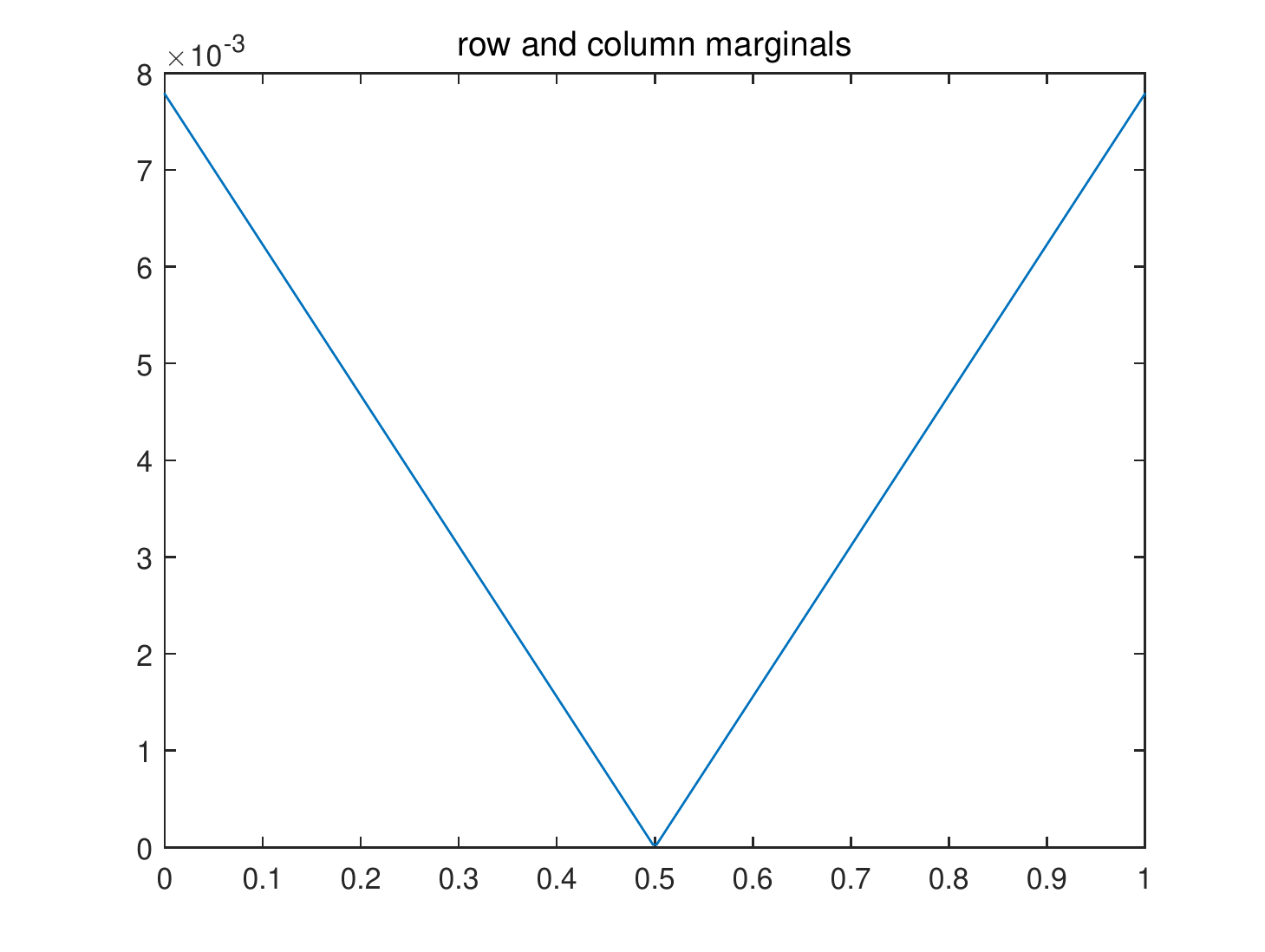}
\end{subfigure}
\end{center}
\caption{Problem data \label{problemdata}}
\end{figure}

\begin{figure}
\begin{subfigure}[b]{.3\linewidth}
\includegraphics[width=5cm]{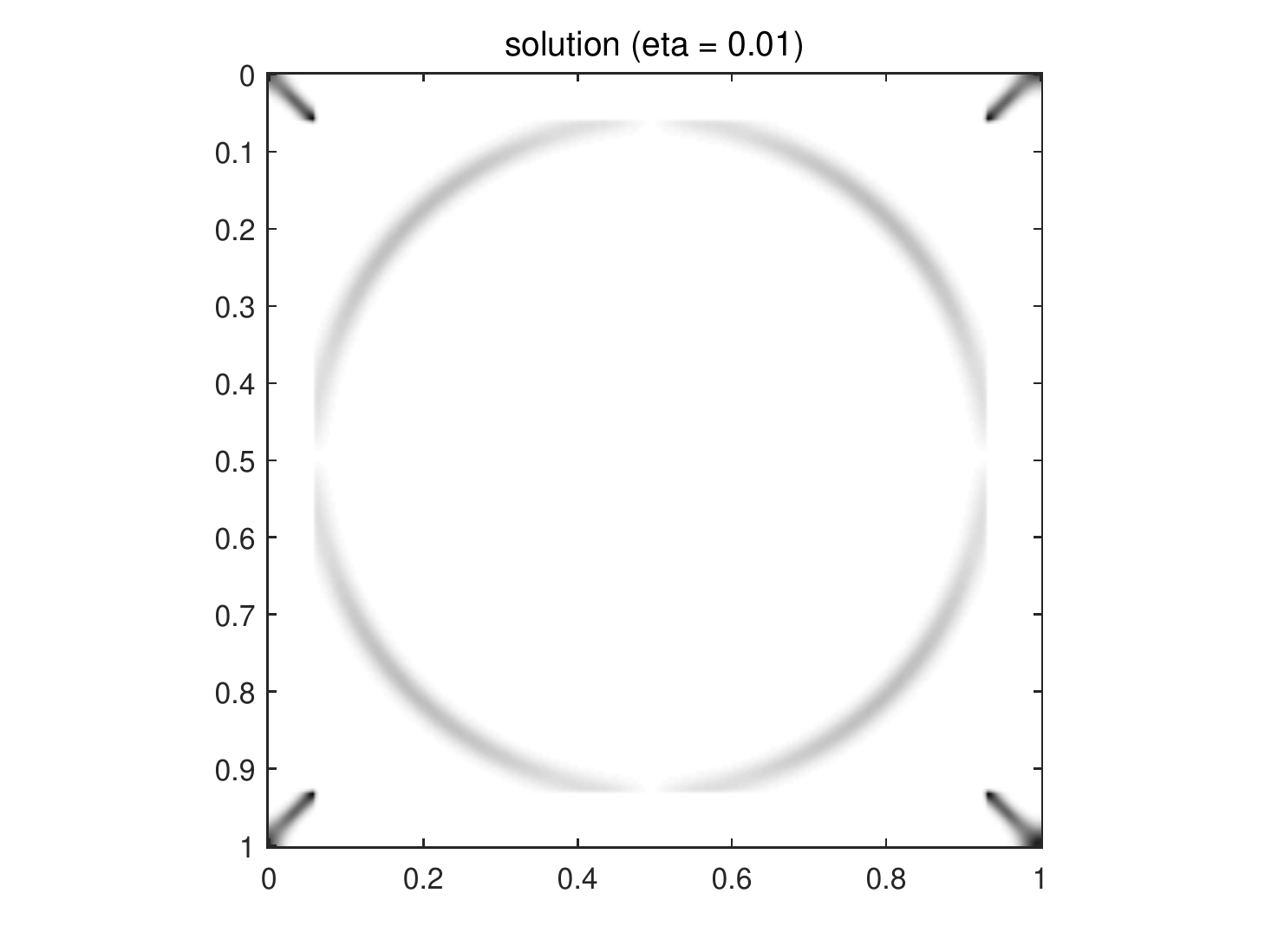}
\end{subfigure}
\begin{subfigure}[b]{.3\linewidth}
\includegraphics[width=5cm]{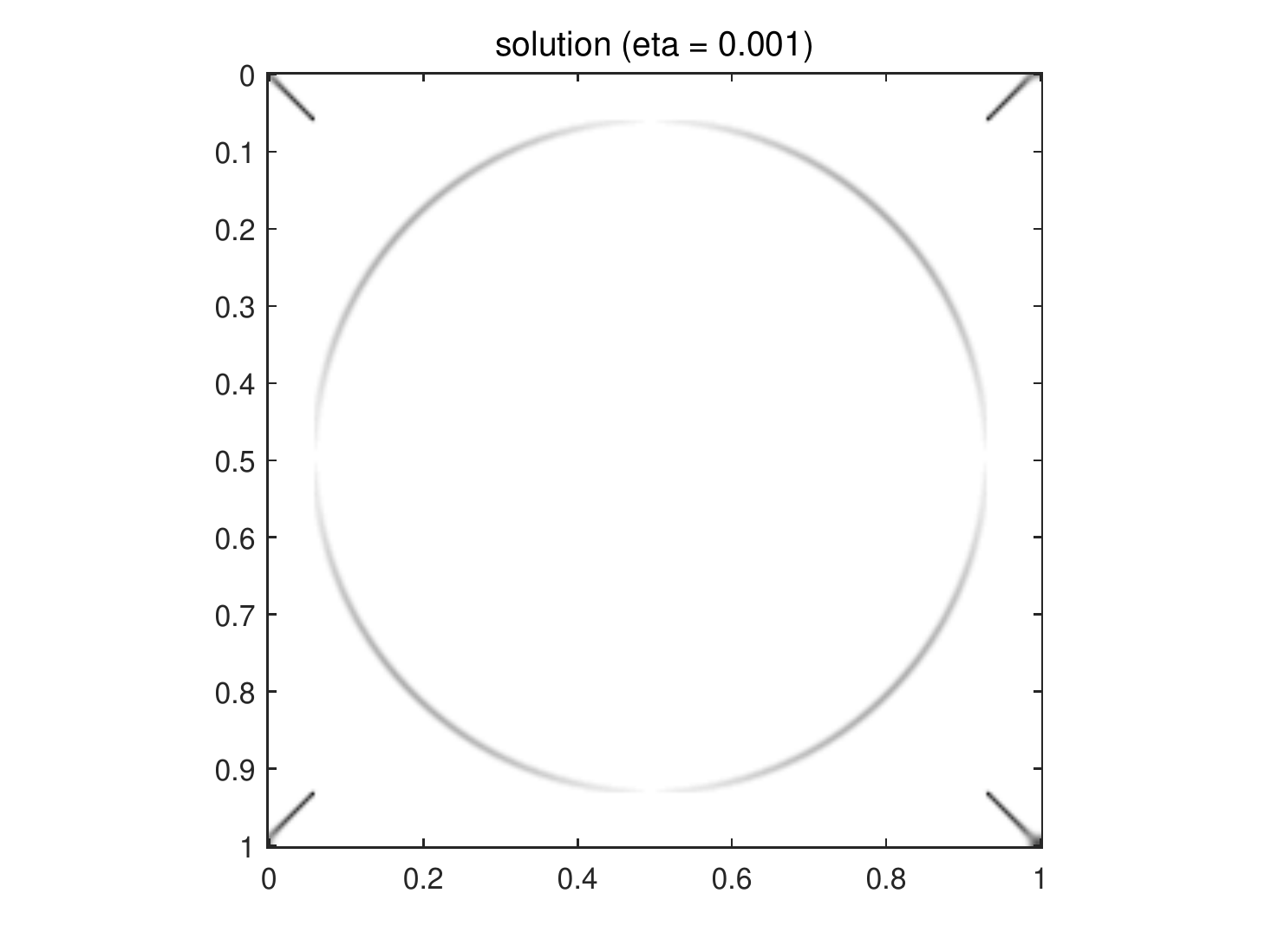}
\end{subfigure}
\begin{subfigure}[b]{.3\linewidth}
\includegraphics[width=5cm]{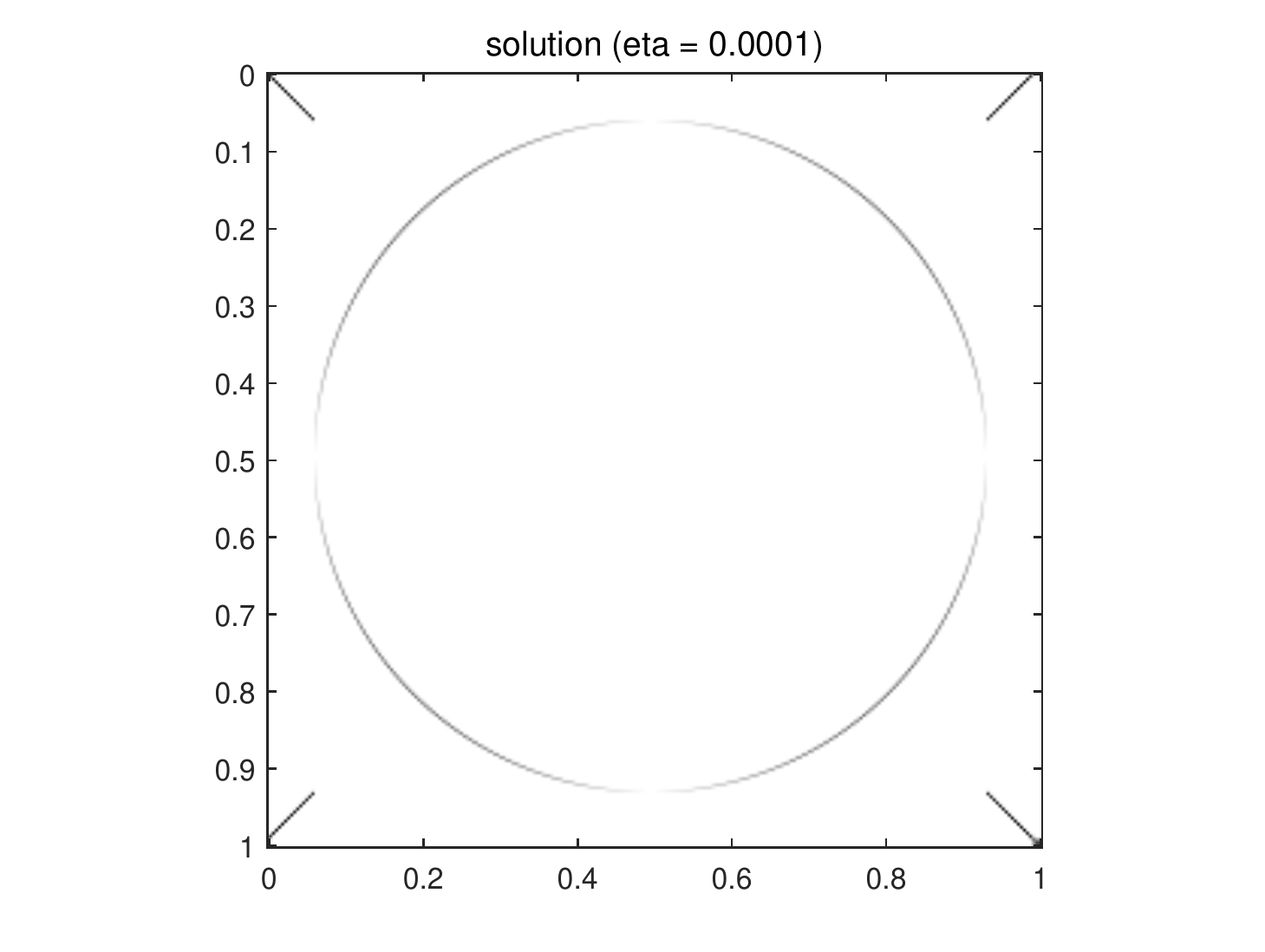}
\end{subfigure}
\\
\begin{subfigure}[b]{.3\linewidth}
\includegraphics[width=5cm]{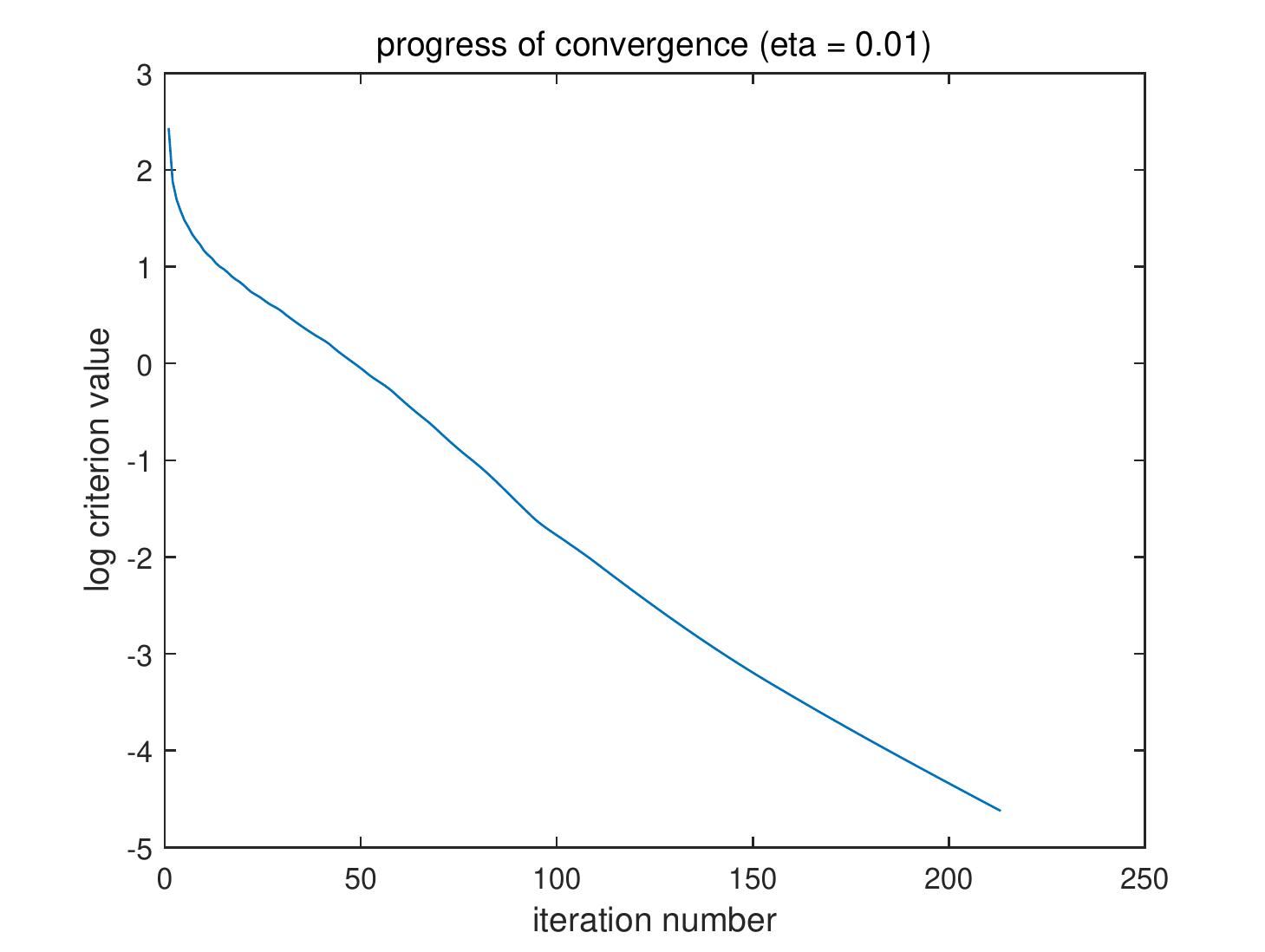}
\end{subfigure}
\begin{subfigure}[b]{.3\linewidth}
\includegraphics[width=5cm]{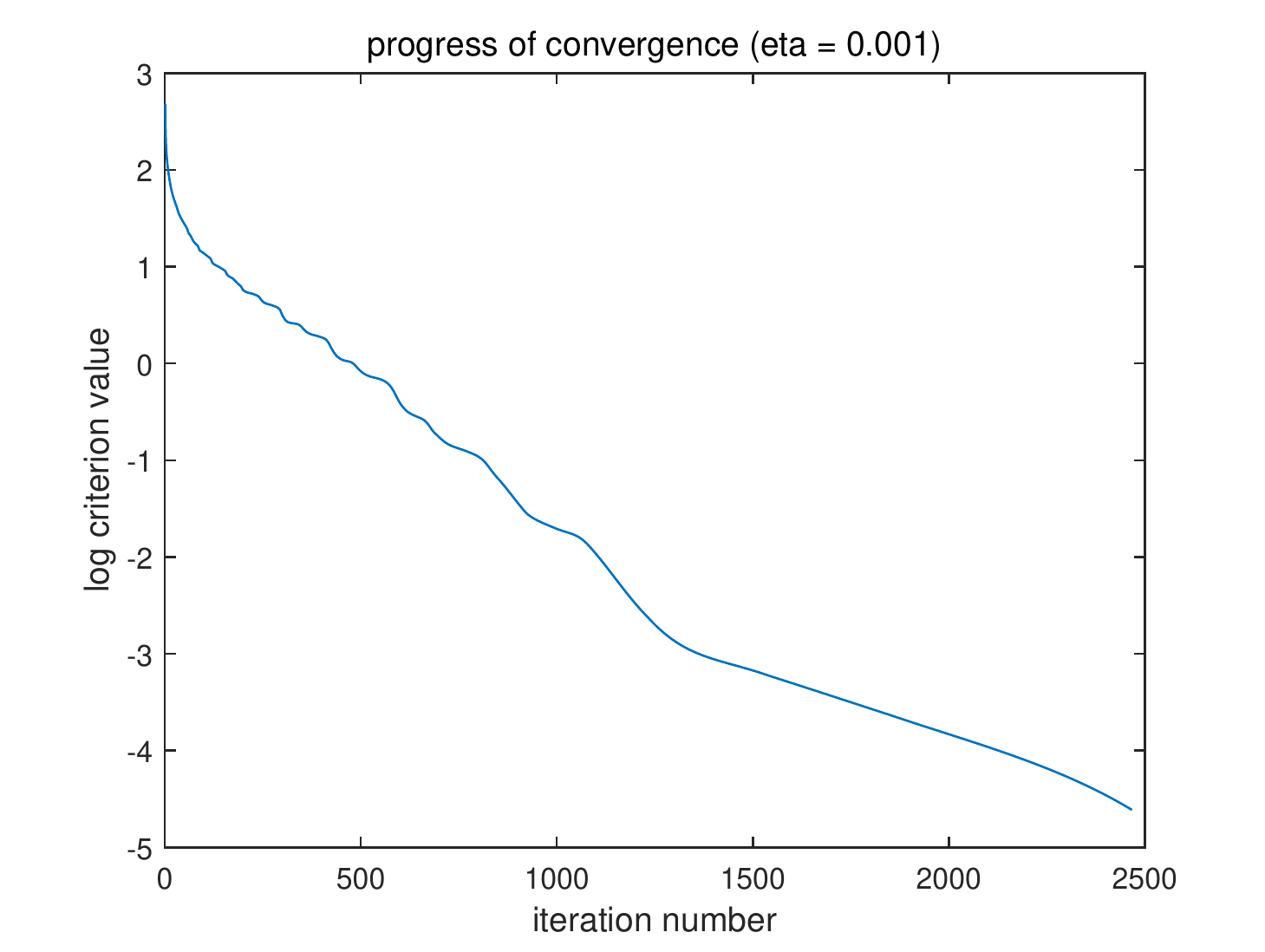}
\end{subfigure}
\begin{subfigure}[b]{.3\linewidth}
\includegraphics[width=5cm]{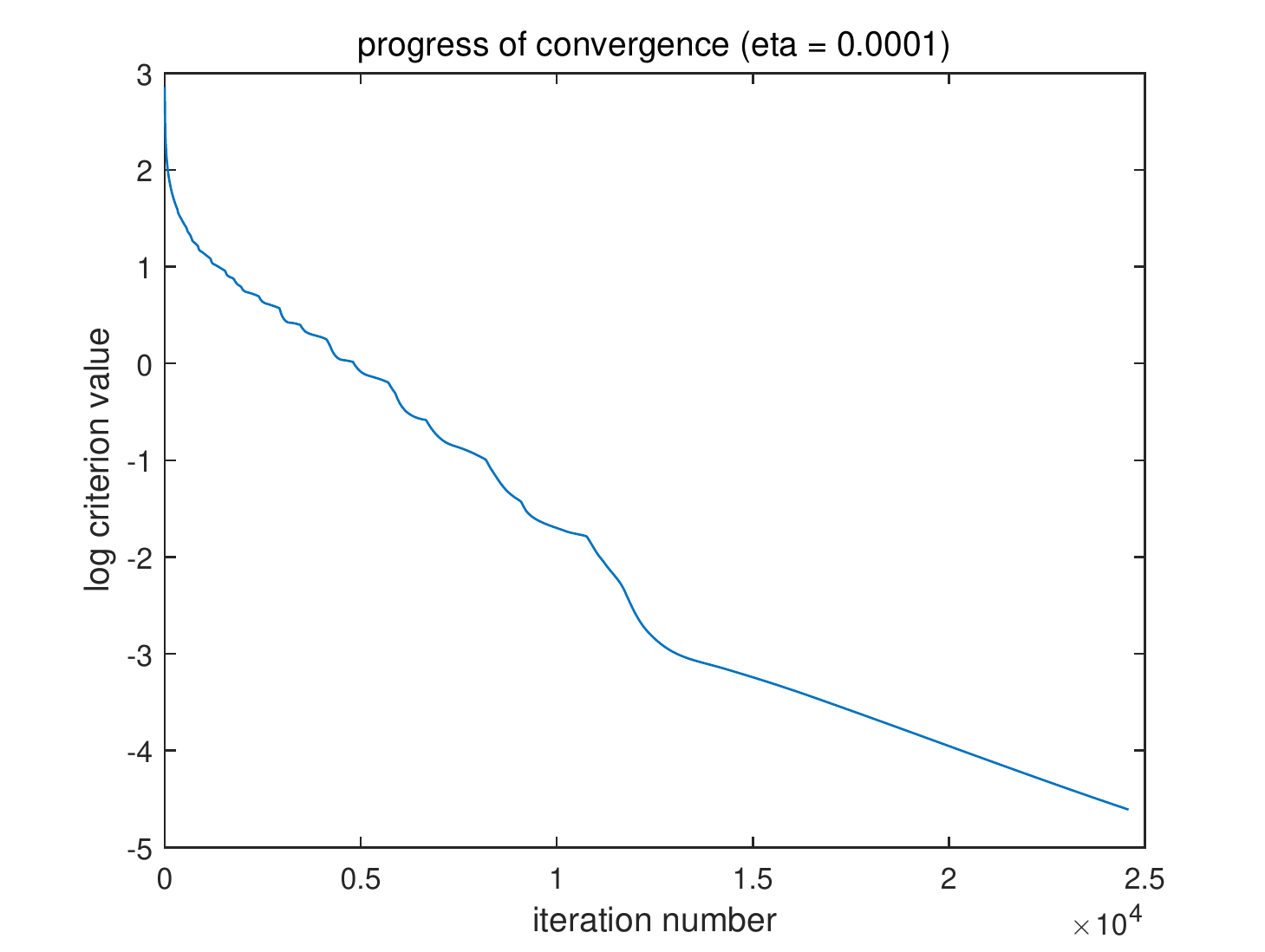}
\end{subfigure}
\caption{Results of computational experiment \label{results}}
\end{figure}

\begin{remark}
The computation time can be reduced by carrying out the iteration in stages, starting with a relatively high value of the regularization parameter, and reducing it stepwise. Such stagewise calculation is standard in the interior point method of linear programming (see for instance \cite{Wright}). A simple reduction scheme was applied to the problem data given above, consisting of twelve stages in which the regularization parameter was reduced
by the factor 1.5 at each step to reach the final level $10^{-4}$, while the tolerance level was kept at $0.01$ in
all stages. The stagewise computation required less than 1000 iterations to converge. Compared to the single-stage calculation that starts immediately with the final
value of the regularization parameter, this implies a reduction of computation time by approximately a factor 25.
\end{remark}

It is seen from the graphs of the criterion value (\ref{critval}) that the convergence is somewhat irregular. This behavior becomes
even more pronounced in the case of small examples. Of course, in such examples there is not really a reason to apply regularization,
since small instances can be easily solved by any linear programming code. However, one may be able to get more insight into the
behavior of the iterative algorithm from simple examples. Consider a problem with the following data:
\begin{equation} \label{smallexdata}
(a_{ij}) = \begin{bmatrix}
0 && 1 && 0.5 \\ 0.7 && 0.5 && 0.3 \\ 0.6 && 0.3 && 0
\end{bmatrix},
\quad r = [ 0.25 \;\; 0.25 \;\; 0.5], \quad c = [0.2 \;\; 0.6 \;\; 0.2].
\end{equation}
\begin{figure}[t]
\begin{center}
\includegraphics[width=8cm]{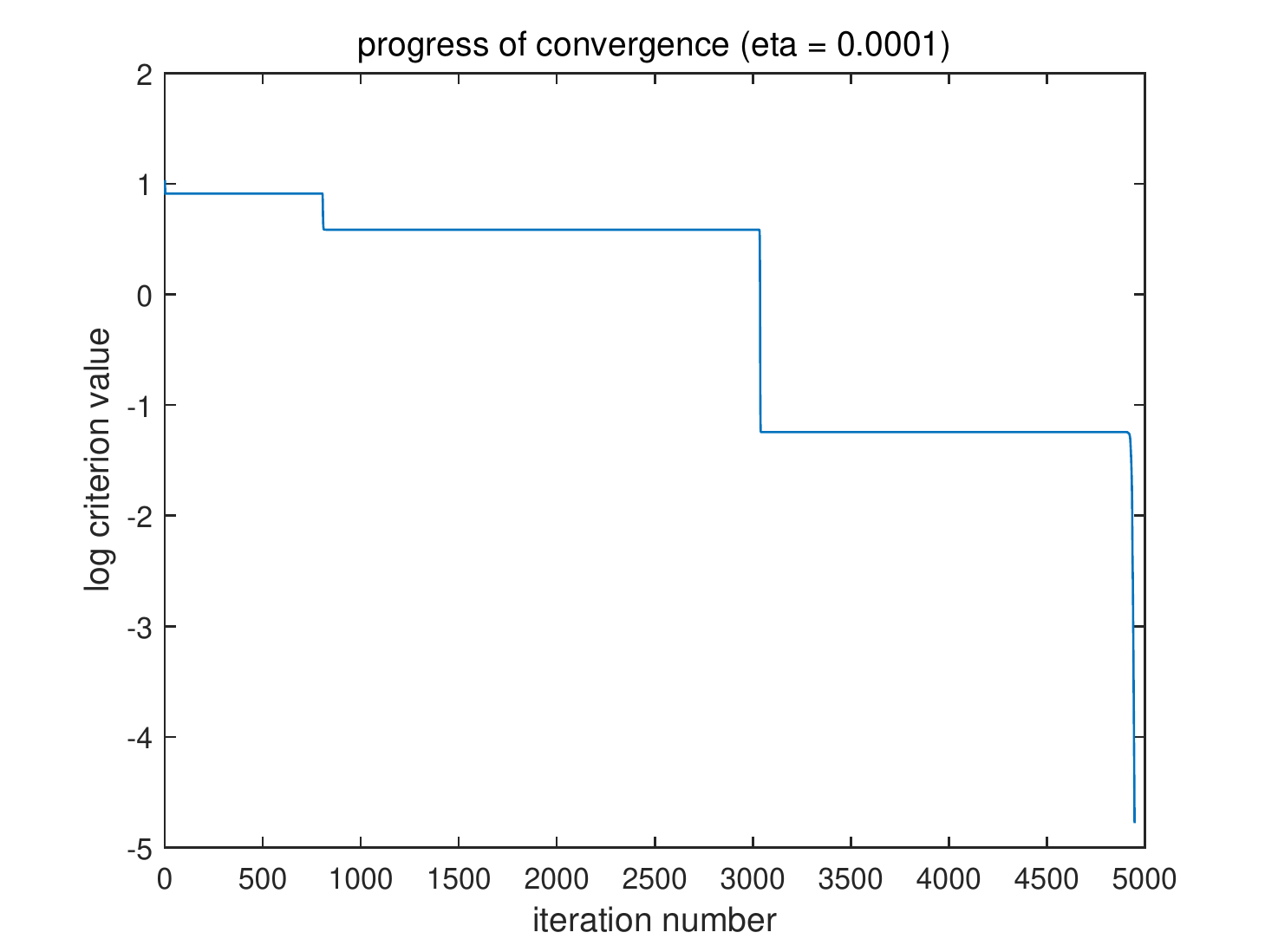}
\end{center}
\caption{Convergence behavior in example (\ref{smallexdata}) \label{smallex}}
\end{figure}
\vskip0mm\noindent
The convergence behavior of the iterative algorithm is shown in Fig.\,\ref{smallex}. As is evident from the figure, the convergence towards the true solution is not smooth at all. There are long periods in which little progress is made, and the algorithm appears almost trapped. On the other hand, once the algorithm is close to the true solution, progress is quick; reducing the tolerance level from $10^{-2}$, as used here, to levels close to machine accuracy would increase the total number of iterations just by a few percent. Upon inspection of the intermediate results, it appears that the phases of stagnation are linked to approximate solutions that, when used as a starting point for the IPFP algorithm, would lead to cycling, due to the presence of too many zeros.\footnote{Precise conditions for convergence of the IPFP algorithm for non-negative weight matrices are given in \cite{Pukelsheim}.} The algorithm passes closely by the following row-balanced matrices in succession:\footnote{More precisely, the matrices given are of the form $(x_{ij}^{(k)}) $ where $x^{(k)}_{ij}$ = $(z^{(k)}_{ij})^{1/\eta}$, and the matrices $\big(z_{ij}^{(k)}\big)$ are the ones that are actually computed in the proposed algorithm.}
$$
\begin{bmatrix}
         0   &&  0.1875   &&  0.0625 \\
    0.25   &&    0          &&   0         \\
    0.5     &&    0          &&   0
\end{bmatrix}
\rightarrow
\begin{bmatrix}
         0   &&  0.25 &&        0 \\
        0    &&    0    && 0.25   \\
      0.5   &&    0    &&     0    \\
\end{bmatrix}
\rightarrow
\begin{bmatrix}
         0    && 0.25  &&       0 \\
         0   && 0 &&    0.25     \\
    0.1875   &&  0.3125    && 0
\end{bmatrix}
\rightarrow
\begin{bmatrix}
         0    && 0.25    &&    0 \\
         0    && 0.05    && 0.2 \\
      0.2    && 0.3      &&  0
\end{bmatrix}.
$$
Each of the first three matrices above has the property that it would cause the IPFP algorithm to cycle; the final matrix corresponds to the true solution. When the regularization parameter is reduced, the algorithm follows essentially the same path, but takes smaller steps, and hence requires more iterations.

\section{Conclusions} \label{conclusions}

The purpose of this paper has been to establish a relationship between the optimal transport problem on the one hand, and a certain class of multi-objective equilibrium problems with linear reward functions on the other hand. This connection forms a natural complement to a similar relationship which exists in the strictly concave / strictly convex case. By the availability of this connection, it is possible to apply different perspectives to a given problem. This variety of viewpoints has been illustrated in the context of regularization, where it was shown that entropic regularization in the single-objective problem corresponds to isoelastic regularization in the multi-objective framework. The ability to switch between different interpretations may be useful in the design of algorithms. The curious behavior of the iterative algorithm for small values of the regularization parameter seems worthy of further investigation.

\bibliographystyle{plainnat}

\begin{thebibliography}{10}
\providecommand{\url}[1]{{#1}}
\providecommand{\urlprefix}{URL }
\expandafter\ifx\csname urlstyle\endcsname\relax
  \providecommand{\doi}[1]{DOI~\discretionary{}{}{}#1}\else
  \providecommand{\doi}{DOI~\discretionary{}{}{}\begingroup
  \urlstyle{rm}\Url}\fi

\bibitem{Rachev1}
Rachev, S.T., R{\"u}schendorf, L.: Mass Transportation Problems. Volume I:
  Theory.
\newblock Springer, New York (1998)

\bibitem{Rachev2}
Rachev, S.T., R{\"u}schendorf, L.: Mass Transportation Problems. Volume II:
  Applications.
\newblock Springer, New York (1998)

\bibitem{Villani03}
Villani, C.: Topics in Optimal Transportation.
\newblock No.~58 in Graduate Studies in Mathematics. American Mathematical
  Society, Providence, RI (2003)

\bibitem{Villani}
Villani, C.: Optimal Transport. Old and New.
\newblock Springer, Berlin (2008)

\bibitem{Monge}
Monge, G.: M\'emoire sur la th\'eorie des d\'eblais et des remblais.
\newblock In: Histoire de l'Acad\'emie Royale des Sciences avec les m\'emoires
  de math\'ematique et de physique tir\'es des registres de cette Acad\'emie,
  pp. 666--705 (1781)

\bibitem{Kantorovich}
Kantorovich, L.V.: On the translocation of masses.
\newblock Doklady Akademii Nauk USSR \textbf{37}, 199--201 (1942).
\newblock In Russian. English translation in \textit{Journal of Mathematical
  Sciences} {\bf 133}, 1381--1382 (2006)

\bibitem{Frechet}
Fr\'echet, M.: Sur les tableaux de corr\'elation dont les marges sont
  donn\'ees.
\newblock Annales de l'Universit\'e de Lyon. 3. s\'er., Sciences. Section A:
  Sciences math\'ematiques et astronomie \textbf{14}, 53--77 (1951)

\bibitem{Santambrogio}
Santambrogio, F.: Optimal Transport for Applied Mathematicians.
\newblock Birkh{\"a}user, New York (2015)

\bibitem{Galichon}
Galichon, A.: Optimal Transport Methods in Economics.
\newblock Princeton University Press, Princeton, NJ (2016)

\bibitem{Gale77}
Gale, D.: Fair division of a random harvest.
\newblock Tech. Rep. ORC 77-21, Operations Research Center, UC Berkeley (1977)

\bibitem{Gale79}
Gale, D., Sobel, J.: Fair division of a random harvest: the finite case.
\newblock In: J.R. Green, J.A. Scheinkman (eds.) General Equilibrium, Growth,
  and Trade. Essays in Honor of Lionel McKenzie, pp. 193--198. Academic Press,
  New York (1979)

\bibitem{Buehlmann79}
B\"{u}hlmann, H., Jewell, W.S.: Optimal risk exchanges.
\newblock ASTIN Bulletin \textbf{10}, 243--262 (1979)

\bibitem{Cuturi}
Cuturi, M.: Sinkhorn distances: Lightspeed computation of optimal transport.
\newblock In: Advances in Neural Information Processing Systems, pp.
  2292--2300. MIT Press, Cambridge, MA (2013)

\bibitem{Benamou}
Benamou, J.D., Carlier, G., Cuturi, M., Nenna, L., Peyr{\'e}, G.: Iterative
  {B}regman projections for regularized transportation problems.
\newblock SIAM Journal on Scientific Computing \textbf{37}, A1111--A1138 (2015)

\bibitem{Brualdi}
Brualdi, R.A., Parter, S.V., Schneider, H.: The diagonal equivalence of a
  nonnegative matrix to a stochastic matrix.
\newblock Journal of Mathematical Analysis and Applications \textbf{16}, 31--50
  (1966)

\bibitem{Lemmens}
Lemmens, B., Nussbaum, R.: Nonlinear Perron-Frobenius Theory.
\newblock Cambridge Tracts in Mathematics, vol.~189. Cambridge University
  Press, Cambridge (2012)

\bibitem{PSW1}
Pazdera, J., Schumacher, J.M., Werker, B.J.M.: The composite iteration
  algorithm for finding efficient and financially fair risk-sharing rules.
\newblock Journal of Mathematical Economics \textbf{72}, 122--133 (2017)

\bibitem{Brams}
Brams, S.J., Taylor, A.D.: Fair Division: From Cake-Cutting to Dispute
  Resolution.
\newblock Cambridge University Press, New York (1996)

\bibitem{Moulin}
Moulin, H.J.: Fair Division and Collective Welfare.
\newblock MIT Press, Cambridge, MA (2003)

\bibitem{Foley}
Foley, D.: Resource allocation and the public sector.
\newblock Yale Economic Essays \textbf{7}, 45--98 (1967)

\bibitem{Nash}
Nash, J.F.: The bargaining problem.
\newblock Econometrica \textbf{18}, 155--162 (1950)

\bibitem{Ehrgott}
Ehrgott, M.: Multicriteria Optimization \rm (2nd ed.).
\newblock Springer, Berlin (2006)

\bibitem{Isermann79}
Isermann, H.: The enumeration of all efficient solutions for a linear
  multiple-objective transportation problem.
\newblock Naval Research Logistics \textbf{26}, 123--139 (1979)

\bibitem{Balasko}
Balasko, Y.: Budget-constrained {P}areto-efficient allocations.
\newblock Journal of Economic Theory \textbf{21}, 359--379 (1979)

\bibitem{Gale82}
Gale, D., Sobel, J.: On optimal distribution of output from a jointly owned
  resource.
\newblock Journal of Mathematical Economics \textbf{9}, 51--59 (1982)

\bibitem{Kleinberg}
Kleinberg, J., Tardos, {\'E}.: Balanced outcomes in social exchange networks.
\newblock In: Proceedings of the Fortieth ACM Symposium on Theory of Computing
  \rm (STOC 2008, held in Victoria, BC, Canada), pp. 295--304. ACM, New York
  (2008)

\bibitem{Myerson}
Myerson, R.B.: Game Theory. Analysis of Conflict.
\newblock Harvard University Press, Cambridge, MA (1991)

\bibitem{Gardner}
Gardner, M.: Aha! Insight.
\newblock Freeman, New York (1978)

\bibitem{Michie}
Michie, D., Spiegelhalter, D.J., Taylor, C.C. (eds.): Machine Learning, Neural
  and Statistical Classification.
\newblock Ellis Horwood, Chichester (1994)

\bibitem{Rotar}
Rotar, V.I.: Actuarial Models. The Mathematics of Insurance.
\newblock Chapman \& Hall/CRC, Boca Raton (2007)

\bibitem{Isermann74}
Isermann, H.: Proper efficiency and the linear vector maximum problem.
\newblock Operations Research \textbf{22}, 189--191 (1974)

\bibitem{Schrijver98}
Schrijver, A.: Theory of Linear and Integer Programming.
\newblock Wiley, Chichester (1998)

\bibitem{Burkard}
Burkard, R.E., Klinz, B., Rudolf, R.: Perspectives of {Monge} properties in
  optimization.
\newblock Discrete Applied Mathematics \textbf{70}, 95--161 (1996)

\bibitem{Bregman}
Bregman, L.M.: The relaxation method of finding the common point of convex sets
  and its application to the solution of problems in convex programming.
\newblock USSR Computational Mathematics and Mathematical Physics \textbf{7},
  200--217 (1967)

\bibitem{Kruithof}
Kruithof, J.: Telefoonverkeersrekening.
\newblock De Ingenieur \textbf{52}, E15--E25 (1937).
\newblock In Dutch. English translation: \emph{Telephone traffic calculus},
  Technical report, Bell Telephone Manufacturing Company, Antwerp, January 1952

\bibitem{Deming}
Deming, W.E., Stephan, F.F.: On a least squares adjustment of a sampled
  frequency table when the expected marginal totals are known.
\newblock The Annals of Mathematical Statistics \textbf{11}, 427--444 (1940)

\bibitem{Brown}
Brown, D.T.: A note on approximations to discrete probability distributions.
\newblock Information and Control \textbf{2}, 386--392 (1959)

\bibitem{Sinkhorn}
Sinkhorn, R.: Diagonal equivalence to matrices with prescribed row and column
  sums.
\newblock The American Mathematical Monthly \textbf{74}, 402--405 (1967)

\bibitem{Ireland}
Ireland, C.T., Kullback, S.: Contingency tables with given marginals.
\newblock Biometrika \textbf{55}, 179--188 (1968)

\bibitem{Fienberg}
Fienberg, S.E.: An iterative procedure for estimation in contingency tables.
\newblock The Annals of Mathematical Statistics \textbf{41}, 907--917 (1970)

\bibitem{Csiszar}
Csisz{\'a}r, I.: I-divergence geometry of probability distributions and
  minimization problems.
\newblock The Annals of Probability \textbf{3}, 146--158 (1975)

\bibitem{Rueschendorf}
R\"{u}schendorf, L.: Convergence of the iterative proportional fitting
  procedure.
\newblock The Annals of Statistics \textbf{23}, 1160--1174 (1995)

\bibitem{Bauschke}
Bauschke, H.H., Lewis, A.S.: Dykstra's algorithm with {Bregman} projections: a
  convergence proof.
\newblock Optimization \textbf{48}, 409--427 (2000)

\bibitem{Oshime}
Oshime, Y.: An extension of {M}orishima's nonlinear {P}erron-{F}robenius
  theorem.
\newblock Journal of Mathematics of Kyoto University \textbf{23}, 803--830
  (1983)

\bibitem{Peyre}
Peyr{\'e}, G.: Entropic approximation of {W}asserstein gradient flows.
\newblock SIAM Journal on Imaging Sciences \textbf{8}, 2323--2351 (2015)

\bibitem{Stephan}
Stephan, F.F.: An iterative method of adjusting sample frequency tables when
  expected marginal totals are known.
\newblock The Annals of Mathematical Statistics \textbf{13}, 166--178 (1942)

\bibitem{SinkhornKnopp}
Sinkhorn, R., Knopp, P.: Concerning nonnegative matrices and doubly stochastic
  matrices.
\newblock Pacific Journal of Mathematics \textbf{21}, 343--348 (1967)

\bibitem{Franklin}
Franklin, J., Lorenz, J.: On the scaling of multidimensional matrices.
\newblock Linear Algebra and its Applications \textbf{114}, 717--735 (1989)

\bibitem{Borwein}
Borwein, J.M., Lewis, A.S., Nussbaum, R.D.: Entropy minimization, {DAD}
  problems, and doubly stochastic kernels.
\newblock Journal of Functional Analysis \textbf{123}, 264--307 (1994)

\bibitem{Georgiou}
Georgiou, T.T., Pavon, M.: Positive contraction mappings for classical and
  quantum {Schr{\"o}dinger} systems.
\newblock Journal of Mathematical Physics \textbf{56}, 033\mbox{}301 1--24
  (2015)

\bibitem{Birkhoff}
Birkhoff, G.: Extensions of {Jentzsch's} theorem.
\newblock Transactions of the American Mathematical Society \textbf{85},
  219--227 (1957)

\bibitem{Bushell}
Bushell, P.J.: Hilbert's metric and positive contraction mappings in a {Banach}
  space.
\newblock Archive for Rational Mechanics and Analysis \textbf{52}, 330--338
  (1973)

\bibitem{Menon}
Menon, M.: Reduction of a matrix with positive elements to a doubly stochastic
  matrix.
\newblock Proceedings of the American Mathematical Society \textbf{18},
  244--247 (1967)

\bibitem{Nadler}
Nadler, S.B.: A note on an iterative test of {E}delstein.
\newblock Canadian Mathematical Bulletin \textbf{15}, 381--386 (1972)

\bibitem{Wright}
Wright, S.J.: Primal-Dual Interior-Point Methods.
\newblock SIAM, Philadelphia, PA (1997)

\bibitem{Pukelsheim}
Pukelsheim, F.: Biproportional scaling of matrices and the iterative
  proportional fitting procedure.
\newblock Annals of Operations Research \textbf{215}, 269--283 (2014)

\end{thebibliography}

\end{document}